\documentclass[preprint,12pt]{elsarticle}

\usepackage{amssymb}
\usepackage{amsmath}
\usepackage{amsthm}

\usepackage{xcolor}
\usepackage{hyperref}
\usepackage{tikz}
\usetikzlibrary{arrows}
\newtheorem{defn}{Definition}
\newtheorem{thm}{Theorem}
\newtheorem{lem}{Lemma}
\newtheorem{exmp}{Example}

\newtheorem{cor}{Corollary}

\newdefinition{rmk}{Remark}

\newcommand{\defeats}{\mathrel{\text{def}}}  
\newcommand{\A}{\mathcal{A}}                
\newcommand{\R}{\mathbb{R}}                 
\newcommand{\T}{\mathbb{T}}                 
\newcommand{\defset}{\mathbf{def}}      
\newcommand{\accset}{\mathbf{acc}}      

\journal{}

\begin{document}

\begin{frontmatter}



\title{Encoding higher-order argumentation frameworks with supports to propositional logic systems}


\author[a]{Shuai Tang\corref{cor1}} 
\ead{TangShuaiMath@outlook.com}

\cortext[cor1]{Corresponding author}

\begin{abstract}
	Argumentation frameworks ($AF$s) have been extensively developed, but existing higher-order bipolar $AF$s suffer from critical limitations: attackers and supporters are restricted to arguments, multi-valued and fuzzy semantics lack unified generalization, and encodings often rely on complex logics with poor interoperability. To address these gaps, this paper proposes a higher-order argumentation framework with supports ($HAFS$), which explicitly allows attacks and supports to act as both targets and sources of interactions. We define a suite of semantics for $HAFS$s, including extension-based semantics, adjacent complete labelling semantics (a 3-valued semantics), and numerical equational semantics ([0,1]-valued semantics). Furthermore, we develop a normal encoding methodology to translate $HAFS$s into propositional logic systems ($\mathcal{PLS}$s): $HAFS$s under complete labelling semantics are encoded into Łukasiewicz's three-valued propositional logic ($\mathcal{PL}_3^L$), and those under equational semantics are encoded into fuzzy $\mathcal{PLS}$s ($\mathcal{PL}_{[0,1]}$) such as Gödel and Product fuzzy logics. We prove model equivalence between $HAFS$s and their encoded logical formulas, establishing the logical foundation of $HAFS$ semantics. Additionally, we investigate the relationships between 3-valued complete semantics and fuzzy equational semantics, showing that models of fuzzy encoded semantics can be transformed into complete semantics models via ternarization, and vice versa for specific t-norms. This work advances the formalization and logical encoding of higher-order bipolar argumentation, enabling seamless integration with lightweight computational solvers and uniform handling of uncertainty.
\end{abstract}




\end{frontmatter}



\section{Introduction}
Since Dung's argumentation frameworks ($DAF$s) \cite{dung1995acceptability} were presented, argumentation frameworks ($AF$s) have been developed significantly. Syntactically, higher-order $AF$s ($HOAF$s) \cite{barringer2005temporal, gabbay2009semantics, baroni2009encompassing, baroni2011afra}, bipolar $AF$s \cite{karacapilidis2001computer, verheij2003deflog, cayrol2005acceptability}, higher-order bipolar $AF$s \cite{boella2010support, cohen2015approach, gottifredi2018characterizing, alfano2024credulous, alfano2024cyclic}, frameworks with sets of attacking arguments ($SETAF$s) \cite{nielsen2006generalization, flouris2019comprehensive}, higher-order-set bipolar $AF$s \cite{cayrol2018structure-based,cayrol2018structure,cayrol2018argumentation} and higher-order set $AF$s \cite{tang2025encodingarg} (which are not specific cases of the higher-order-set bipolar $AF$s introduced in \cite{cayrol2018structure-based,cayrol2018structure,cayrol2018argumentation}) have been proposed. Semantically, the labelling semantics \cite{caminada2006issue} advances Dung's extension semantics \cite{dung1995acceptability}. Numerical labelling semantics have been introduced from 3-valued semantics \cite{gabbay2016degrees} to [0,1]-valued semantics, such as equational semantics \cite{gabbay2011introducing, gabbay2012equational} and gradual semantics \cite{barringer2005temporal, cayrol2005graduality, gabbay2015equilibrium, amgoud2022evaluation, amgoud2024higher}.

As a reasoning approach on conflict information, $AF$s are closely connected with logic systems. The research in \cite{besnard2004checking} bridges $DAF$s and the 2-valued propositional logic system. 
An encoding method utilizing quantified Boolean formulas (QBF) is presented in \cite{egly2006reasoning, arieli2013qbf}. The study \cite{egly2006reasoning} extends $DAF$s into frameworks built upon a deductive system's induced derivability relation, proposing a general encoding function for this relation; the encoding technique from \cite{besnard2004checking} is then shown to be a specific instance of this general scheme. Meanwhile, \cite{arieli2013qbf} devises a QBF-based encoding linked to a modified Kleene three-valued logic through the signed theory generated by a $DAF$. Connections between modal logics and $DAF$s are explored in \cite{gabbay2009modal,caminada2009logical,grossi2010logic,villata2012logic, gabbay2015attack}. Furthermore, relations between $DAF$s and intuitionistic logic systems are established in \cite{gabbay2016attack, fandinno2018constructive}.
Integration of $DAF$s with fuzzy set theory via extension-based methods is investigated in \cite{janssen2008fuzzy, wu2016godel, wu2016some, wang2020dynamics, wu2020godel}. Alternatively, labelling-based approaches combined with iterative algorithms are employed to combine $DAF$s with fuzzy set theory in \cite{pereira2011changing, zhao2022efficient}.
Although \cite{dyrkolbotn2014formal} demonstrates a meta-logical equivalence between skeptical reasoning under argumentation semantics and {\L}ukasiewicz's three-valued logic ($\mathcal{PL}_3^L$), it does not furnish an actual encoding from argumentation frameworks into $\mathcal{PL}_3^L$ theories.
A unified adaptive logic framework is introduced in \cite{strasser2014adaptive}, which encodes argumentation frameworks into logical premises and uses a dynamic proof system to model all standard semantics uniformly for both skeptical and credulous reasoning.
The research in \cite{tang2025encoding} links $DAF$s under complete semantics to three-valued propositional logics ($\mathcal{PL}_3$s) and connects $DAF$s under equational semantics to fuzzy propositional logics ($\mathcal{PL}_{[0,1]}$s) by encoding $DAF$s into these logical systems. From the perspective of abstract dialectical frameworks ($ADF$s) based on classical binary propositional logic ($\mathcal{PL}_2$), the expressiveness of logical encodings for $SETAF$s is examined in \cite{dvovrak2023expressiveness}. Several works address logical encodings involving first-order logic for higher-order and bipolar $AF$s \cite{cayrol2017logical,claudette2018logical,cayrol2020logical, lagasquie2021evidential, lagasquie2021necessary, lagasquie2023handling,besnard2023generic}. The method proposed in \cite{tang2025encoding} is extended to higher-hierarchy $AF$s in \cite{tang2025encodingarg}.

Nevertheless, current studies concerning the logical encoding of argumentation frameworks continue to encounter notable deficiencies and obstacles that impede both their practical utility and theoretical coherence. Firstly, the predominant encoding techniques for Argumentation Frameworks ($AF$s) featuring higher-order constructs or collective attackers are predominantly grounded in first-order logic, modal logic, or second-order propositional logic ($\mathcal{PL}_2$). Such techniques frequently incur substantial syntactic overhead—exemplified by quantifiers and intricate predicates in first-order logic—which hinders seamless interoperability with lightweight computational solvers. Alternatively, they exhibit inherent limitations in representing uncertainty and graded truth values; for instance, $\mathcal{PL}_2$ lacks direct mechanisms for capturing fuzzy degrees of acceptability.
Secondly, although prior works have introduced structural formulations for higher-order bipolar $AF$s, a pivotal shortcoming persists: the constituents eligible to act as attackers or supporters are confined solely to arguments. From a standpoint of theoretical uniformity and real-world applicability, attacks and supports are themselves entities susceptible to being targeted and reinforced, much like arguments. Consequently, they should possess the analogous capacity to originate attacks and supports. Existing formalisms are thus incapable of modeling situations in which attacks and supports themselves constitute attackers or supporters, thereby limiting their applicability to a broad spectrum of intricate dialectical exchanges. Thirdly, extant research offers underdeveloped pathways for generalizing between multi‑valued and fuzzy semantics. Many encodings terminate at three‑valued semantics without progressing to continuous fuzzy interpretations, or necessitate substantial structural alterations when shifting across semantic paradigms, thereby precluding a smooth and uniform extension. 
In light of these constraints, the present paper is motivated by three principal objectives, seeking to bridge the identified theoretical voids and meet practical necessities:
\begin{itemize}
	\item To advance the $HOAFN$ into a \emph{higher‑order argumentation framework with supporters} ($HAFS$), thereby remedying a key limitation of existing higher‑order bipolar structures: the constraint that attackers and supporters must be arguments, to the exclusion of attacks and supports. The proposed $HAFS$ explicitly permits attacks and supports to serve as both targets and sources of attack and support.
	\item To introduce a coherent and expressive semantic suite for $HAFS$s, encompassing extension‑based semantics, complete labelling semantics, and numerical equational semantics. Whereas existing semantics for higher‑order bipolar $AF$s are largely confined to discrete labels (e.g., \textit{in}, \textit{out}, \textit{undecided}), the semantics proposed herein establish a unified underpinning that naturally integrates discrete three‑valued reasoning with continuous fuzzy reasoning.
	\item To devise a tailored encoding methodology and a formal equational framework for $HAFS$s, creating a bridge to propositional logic systems ($\mathcal{PLS}$s).
\end{itemize}
This study advances the encoding methodology from $AF$s with only attacks \cite{tang2025encoding,tang2025encodingarg} to higher-order argumentation frameworks with supporters ($HAFS$s). It establishes the logical foundations for $HAFS$s by proving the model equivalence between an $HAFS$ and its encoded logical formula, and introduces a formal equational approach. In summary, the core contributions of this paper are threefold:
\begin{itemize}
	\item \textbf{Syntactic advancement}: We propose the $HAFS$ as a syntactic extension of the $HOAFN$, which enables mutual interactions among attacks and supports (both being targeted and targeting others).
	\item \textbf{Semantic expansion}: We define extension-based semantics and complete labelling semantics for $HAFS$s, and further design a numerical equational semantics for such frameworks.
	\item \textbf{Encoding development}: We provide translations from $HAFS$s under complete semantics into $\mathcal{PL}_3^L$, and from $HAFS$s under equational semantics into $\mathcal{PL}_{[0,1]}$ systems, while also exploring the relationship between general real-valued equational semantics and continuous fuzzy normal encoded equational semantics.
\end{itemize}

\section{Preliminaries}
\subsection{The syntax and semantics of $HOAFN$s}
	\begin{defn}[Higher-Order $AF$ with Necessities ($HOAFN$)]\cite{alfano2024credulous}
		A higher-order argumentation framework with necessities ($HOAFN$) is a tuple 
		$\langle \mathcal{A}, \mathbb{R},$ $\mathbb{T} \rangle$, where 
		$\mathcal{A}$ is a finite set of arguments, 
		$\mathbb{R}$ is a finite set of attacks, 
		$\mathbb{T}$ is a finite set of (necessary) supports, and 
		$\mathbb{R} \cup \mathbb{T} \subseteq \mathcal{A} \times (\mathcal{A} \cup \mathbb{R} \cup \mathbb{T})$.
	\end{defn}

\begin{defn}[$ASAF$ defeats]\cite{alfano2024credulous}
	Let $\langle \A, \R, \T \rangle$ be an $ASAF$, $\vartheta \in (\A \cup \R \cup \T)$, $\alpha \in \R$, and $S \subseteq (\A \cup \R \cup \T)$. We say that $\alpha$ defeats $\vartheta$ given $S$ (denoted as $\alpha \defeats_S \vartheta$) if either: 
	\begin{enumerate}
		\item $t(\alpha) = \vartheta$, 
		\item $\vartheta \in \R$ and $t(\alpha) = s(\vartheta)$, 
		\item there exists $b \in \A$ such that $t(\alpha) = b$ and $(b, \vartheta) \in (\T \cap S)^+$, or 
		\item $\vartheta \in \R$ and there exists $b \in \A$ such that $t(\alpha) = b$ and $(b, s(\vartheta)) \in (\T \cap S)^+$.
	\end{enumerate}
\end{defn}

\begin{defn}[$ASAF$ defeated and acceptable sets]\cite{alfano2024credulous}
	Let $\langle \A, \R, \T \rangle$ be an $ASAF$ and $S \subseteq (\A \cup \R \cup \T)$. We define:
	\begin{itemize}
		\item $\mathbf{def}(S) = \left\{ \vartheta \in (\A \cup \R \cup \T) \mid \exists \alpha \in (\R \cap S): \alpha \defeats_S \vartheta \right\}$.  
		\item $\mathbf{acc}(S) = \left\{ \vartheta \in (\A \cup \R \cup \T) \mid \forall \alpha \in \R: \alpha \defeats_S \vartheta \implies \alpha \in \mathbf{def}(S) \right\}$.
	\end{itemize}
\end{defn}

\begin{defn}[$RAFN$ defeats]\cite{alfano2024credulous}
	Let $\langle \A, \R, \T \rangle$ be a $RAFN$, $\vartheta \in (\A \cup \R \cup \T)$, $a \in \A$, and $S \subseteq (\A \cup \R \cup \T)$. We say that argument $a$ defeats $\vartheta$ given $S$ (denoted as $a \defeats_S \vartheta$) if either: 
	\begin{enumerate}
		\item $(a, \vartheta) \in (\R \cap S)$, or 
		\item there exists $b \in \A$ such that $(a, b) \in (\R \cap S)$ and $(b, \vartheta) \in (\T \cap S)^+$.
	\end{enumerate}
\end{defn}

\begin{defn}[$RAFN$ defeated and acceptable sets]\cite{alfano2024credulous}
	Let $\langle \A, \R, \T \rangle$ be a $RAFN$ and $S \subseteq (\A \cup \R \cup \T)$. We define:
	\begin{itemize}
		\item $\mathbf{def}(S) = \left\{ \vartheta \in (\A \cup \R \cup \T) \mid \exists b \in \A \cap S: b \defeats_S \vartheta \right\}$;  
		\item $\mathbf{acc}(S) = \left\{ \vartheta \in (\A \cup \R \cup \T) \mid \forall b \in \A: b \defeats_S \vartheta \implies b \in \mathbf{def}(S) \right\}$.
	\end{itemize}
\end{defn}

From the above definitions and logical considerations, if $\nexists b \in \A$ s.t. $\ b \defeats_S \vartheta$ then $\vartheta\in \mathbf{acc}(S)$.

\begin{defn}[$HOAFN$ semantic notions and extensions]\cite{alfano2024credulous}
	Let $\Delta = \langle \A, \R, \T \rangle$ be a $HOAFN$ and $S \subseteq (\A \cup \R \cup \T)$. We say that $S$ is:
	\begin{itemize}
		\item \emph{conflict-free} iff $S \cap \defset(S) = \emptyset$;  
		\item \emph{admissible} iff it is conflict-free and $S \subseteq \accset(S)$;  
		\item \emph{a complete extension of $\Delta$} iff it is conflict-free and $S = \accset(S)$;  
		\item \emph{a preferred extension of $\Delta$} iff it is a maximal (w.r.t. $\subseteq$) complete extension;  
		\item \emph{a stable extension of $\Delta$} iff it is a preferred extension such that $S \cup \defset(S) = (\A \cup \R \cup \T)$;  
		\item \emph{the grounded extension of $\Delta$} iff it is the smallest (w.r.t. $\subseteq$) complete extension.  
	\end{itemize}
\end{defn}

\begin{defn}[extension-based labellings]\cite{alfano2024credulous}
	Let $\Delta = \langle \A, \R, \T \rangle$ be an $HOAFN$ and $E$ an $\sigma$-extension of $\Delta$, with $\sigma \in \{co, gr, st, pr\}$. A labeling of $\Delta$ under semantics $\sigma$ is a function $L_E : (\A \cup \R \cup \T) \to \{ \mathrm{in}, \mathrm{out}, \mathrm{undec}\}$ such that for any $\vartheta \in (\A \cup \R \cup \T)$ it holds that:
	\begin{enumerate}
		\item $L_E(\vartheta) = \mathrm{in}$ iff $\vartheta \in \accset(E)$,  
		\item $L_E(\vartheta) = \mathrm{out}$ iff $\vartheta \in \defset(E)$, and  
		\item $L_E(\vartheta) = \mathrm{undec}$ iff $\vartheta \in (\A \cup \R \cup \T) \setminus (\accset(E) \cup \defset(E))$.  
	\end{enumerate}
\end{defn}

\subsection{Propositional logic systems}
Some essential knowledge of propositional logic systems is given in \cite{tang2025encoding,tang2025encodingarg}. For in-depth knowledge of $\mathcal{PLS}$s, interested readers may consult \cite{klir1995fuzzy,Hajek1998,Klement2000Triangular, bergmann2008introduction, belohlavek2017fuzzy}. In this subsection, we only present some foundational operations.

For the meta-language of this paper, we denote ``$A$ if and only if $B$'' by ``$A \Longleftrightarrow B$'' and ``if $B$ then $A$'' by ``$B \Longrightarrow A$'', where $A$ and $B$ are sentences. 

A (formal) \textbf{language} $\mathcal{L}$ of a $\mathcal{PLS}$ is determined by the following components:
	
	\begin{enumerate}
		\item \textbf{Alphabet:} The alphabet of $\mathcal{L}$ consists of:
		\begin{itemize}
			\item A countable set $\text{Prop} = \{p_1, p_2, p_3, \dots\}$ of \textit{propositional variables}.
			\item The set of \textit{logical connectives}: $\{\neg, \wedge, \vee, \rightarrow, \leftrightarrow\}$.
			\begin{itemize}
				\item $\neg$ is a unary connective (negation).
				\item $\wedge, \vee, \rightarrow, \leftrightarrow$ are binary connectives (conjunction, disjunction, implication, equivalence).
			\end{itemize}
			\item The auxiliary symbols: parentheses “(” and “)”.
		\end{itemize}
		
		\item \textbf{Syntax (Formation Rules):} The set $\text{Form}(\mathcal{L})$ of \textit{well-formed formulas} (wffs) is the smallest set such that:
		\begin{itemize}
			\item If $p \in \text{Prop}$, then $p \in \text{Form}(\mathcal{L})$.
			\item If $\varphi \in \text{Form}(\mathcal{L})$, then $(\neg \varphi) \in \text{Form}(\mathcal{L})$.
			\item If $\varphi, \psi \in \text{Form}(\mathcal{L})$ and $\circ \in \{\wedge, \vee, \rightarrow, \leftrightarrow\}$, then $(\varphi \circ \psi) \in \text{Form}(\mathcal{L})$.
		\end{itemize}
		
		\item \textbf{Additional Conventions:}
		\begin{itemize}
			\item Outer parentheses are often omitted.
			\item The precedence order of connectives (from highest to lowest) is: $\neg$, $\wedge$, $\vee$, $\rightarrow$, $\leftrightarrow$.
			\item The symbols $\top$ (tautology) and $\bot$ (contradiction) may be added as atomic formulas, or defined as $p \vee \neg p$ and $p \wedge \neg p$, respectively.
		\end{itemize}
	\end{enumerate}
	
	In this paper, we only consider any $\mathcal{PLS}$ that has numerical assignment domains (subsets of $[0,1]$) and their corresponding numerical semantic systems, i.e., the semantics of a $\mathcal{PLS}$ is decided by the interpreted truth-functions (operators) of the connectives in the $\mathcal{PLS}$. We use $\|\cdot\|$ to represent an assignment or a model over the Prop or the evaluation of all formulas. In each $\mathcal{PLS}$ and any $\|\cdot\|$, we always have that $\|\bot\|=0$ and $a \leftrightarrow b \;:=\; (a \rightarrow b) \wedge (b \rightarrow a)$ for any formulas $a$ and $b$. 
	
In the $\mathcal{PL}_3^L$ \cite{Lukasiewicz1970ThreeValued}, the connectives $\neg, \wedge$ and $\vee$ are interpreted as:
\begin{itemize}
	\item $\|\neg a\|=1-\|a\|$
	\item $\|a\wedge b\|=\min\{\|a\|, \|b\|\}$
	\item $\|a\vee b\|=\max\{\|a\|, \|b\|\}$
\end{itemize}
where  $a, b$ are propositional formulas. 
The connective $\rightarrow$ is interpreted by the operator $\Rightarrow$ such that $\|a\rightarrow b\| = \|a\|\Rightarrow\|b\|$, where the truth-degree table for $\Rightarrow$ is given in Table~\ref{Rightarrow}.

\begin{table}[htbp]
	\centering
	\begin{tabular}{c | c c c}
		$\Rightarrow$ & 0 & $\frac12$ & 1 \\ 
		\hline
		0 & 1 & 1 & 1 \\
		$\frac12$ & $\frac12$ & 1 & 1 \\
		1 & 0 & $\frac12$ & 1
	\end{tabular}
	\caption{Truth‑degree table of $\Rightarrow$ in $\mathcal{PL}_3^L$.}
	\label{Rightarrow}
\end{table}
Because $a\leftrightarrow b \;:=\; (a\rightarrow b)\wedge(b\rightarrow a)$, we obtain
\begin{equation*}
	\|a\leftrightarrow b\| 
	= \min\{\|a\rightarrow b\|,\;\|b\rightarrow a\|\} 
	= \min\{\|a\|\Rightarrow\|b\|,\;\|b\|\Rightarrow\|a\|\}.
\end{equation*}
Hence $\leftrightarrow$ is interpreted by the operator $\Leftrightarrow$, where the truth-degree table for $\Leftrightarrow$ is shown in Table~\ref{Leftrightarrow}.
\begin{table}[htbp]
	\centering
	\begin{tabular}{c | c c c}
		$\Leftrightarrow$ & 0 & $\frac12$ & 1 \\ 
		\hline
		0 & 1 & $\frac12$ & 0 \\
		$\frac12$ & $\frac12$ & 1 & $\frac12$ \\
		1 & 0 & $\frac12$ & 1
	\end{tabular}
	\caption{Truth‑degree table of $\Leftrightarrow$ in $\mathcal{PL}_3^L$.}
	\label{Leftrightarrow}
\end{table}

Next, we review some important fuzzy operators in $\mathcal{PL}_{[0, 1]}$s \cite{klir1995fuzzy,Hajek1998, Klement2000Triangular}. The connectives $\neg, \wedge, \rightarrow$ are interpreted as a negation, a t-norm, and an implication, respectively.
A negation $N$ is called a \emph{standard negation} if $\forall m\in[0, 1]$: $N(m)=1-m$. The following are key continuous t-norms and their corresponding $R$-implications ($m, n\in[0, 1]$):
\begin{itemize}
	\item G\"{o}del t-norm: $T_G (m, n)=\min\{m, n\}$;
	\item R-implication of $T_G$: $I_G(m, n) = 
	\begin{cases}
		1 & \quad m\leq n \\
		n & \quad m>n 
	\end{cases};$
	\item {\L}ukasiewicz t-norm: $T_L (m, n)= \max\{0, m+n-1\}$;
	\item R-implication of $T_L$: $I_L(m, n) = \min\{1 - m +n, 1\};$
	\item Product t-norm: $T_P (m, n) = m\cdot n$;
	\item R-implication of $T_P$: $I_P(m, n)=\begin{cases}
		1 & \quad m\leq n \\
		\frac{n}{m} & \quad m>n
	\end{cases}$.
\end{itemize}
Note that for any R-implication $I$, we have that $I(m,n)=I(n,m)=1$ iff $m=n$, where $m,n\in[0,1]$. Therefore, $\|a\leftrightarrow b\|=1$ iff $\|a\leftarrow b\|=\|a\rightarrow b\|=1$ iff $\|a\|=\|b\|$, where $a$ and $b$ are propositional formulas.

In this paper, we use $\mathcal{PL}_{[0,1]}^G$ \cite{esteva2000residuated} to denote the $\mathcal{PLS}$ equipped with the standard negation,  G\"{o}del t-norm $T_G$ and R-implication $I_G$. Similarly, $\mathcal{PL}_{[0,1]}^P$ \cite{esteva2000residuated} denotes the $\mathcal{PLS}$ equipped with the standard negation,  Product t-norm $T_P$ and R-implication $I_P$, while $\mathcal{PL}_{[0,1]}^L$ \cite{esteva2000residuated} denotes the $\mathcal{PLS}$ equipped with the standard negation, {\L}ukasiewicz t-norm $T_L$ and R-implication $I_L$. 

\section{Motivations}
In the Preliminary Section, we have presented the general motivations. In this section, we will elaborate on the motivations of this paper.

We first analyze the syntactic limitations of current higher-order bipolar $AF$s. 
While $RAFN$s were introduced in \cite{cayrol2018structure-based,cayrol2018structure}, $REBAF$s in \cite{cayrol2018argumentation}, and $ASAF$s in \cite{cohen2015approach} (later developed in \cite{gottifredi2018characterizing} and summarized in \cite{alfano2024credulous}), 
all three frameworks share a common restriction: they do not permit attacks or supports themselves to be targeted by other relations. 
However, since attacks or supports can be targeted, attacks or supports should also have the ability to target others. Moreover, in practical argumentation scenarios, situations exist where attacks or supports do serve as targets. Thus, we propose the novel $HAFS$ whose syntax explicitly accommodates this case.

Then we analyze the semantic perspective. In \cite{alfano2024credulous}, the authors simplify the syntax of $RAFN$s \cite{cayrol2018structure-based,cayrol2018structure} as the same syntax of $ASAF$s and name the two higher-order bipolar $AF$s as $HOAFN$s. When $RAFN$s \cite{alfano2024credulous} and $ASAF$s are restricted as $AF$s that have only support relations, they have the same semantics. The difference between both frameworks results from the semantics on attack relations. For the $ASAF$, the most serious semantic shortcoming is that the value of an attack depends on the value of its source. If the value of an attack's source is 0, then the value of the attack must be 0. However, on one hand, since attacks represent logical attack-relationships between arguments (or attacks or supports), the truth of a logical attack-relationship will not depend on the truth of its source—just like in classical propositional logic, the validity of an implication proposition does not depend on the validity of its antecedent. On the other hand, we encounter many situations where the source of an attack is invalid, yet the attack itself is valid. 
Thus, the value of an attack should not depend on the value of its source. This semantic independence-feature is characterized by the $RAFN$s \cite{alfano2024credulous}. In fact, when we restrict a $RAFN$ \cite{alfano2024credulous} as an $AF$ only with the attack relation, we obtain an $HLAF$ \cite{gabbay2009semantics} and its associated semantics. Therefore, in this paper we tend to choose the semantics of $RAFN$s \cite{alfano2024credulous} to generalize it as the semantics of $HAFS$s.

Another motivation stems from the treatment of support cycles in higher-order bipolar $AF$s. In \cite{alfano2024credulous}, the authors avoid discussing the $HOAFN$s with support cycles. While the authors in \cite{alfano2024cyclic} introduce a semantics for higher-order bipolar $AF$s that admits support cycles, their approach deems an argument (or an attack or a support) unacceptable whenever it participates in a valid support cycle. In our view, extensions that include such an element should remain possible, as it may well be justified. Therefore, we aim to develop a novel semantics for higher-order bipolar $AF$s with support cycles, aligning with this perspective. 

Numerical $[0,1]$-valued semantics for higher-order or set $AF$s are developed in \cite{amgoud2024higher, yun2020ranking, yun2021gradual, tang2025encodingarg},
but the numerical semantics for higher-order bipolar $AF$s are absent.
Thus, this absence motivates us to develop the numerical semantics for $HAFS$s. We promote the encoding approach to develop the semantics of $HAFS$s, since this approach connects $AF$s and logic systems, which provides logical foundations for argumentation semantics and avoid arbitrary diverse semantic definitions.

\section{The syntax and basic semantics of $HAFS$s}
In this section, we present the syntax and basic semantics of $HAFS$s. First, we formalize the syntax of $HAFS$s. We then introduce the adjacent complete labelling semantics --- the truth degree of each argument (or attack or support) is decided only by the truth degrees of its adjacent elements --- as a foundational and concise semantics for $HAFS$s. To situate $HAFS$s within traditional argumentation semantic paradigms, we further develop its extension-based semantics and define the extension-derived complete labelling. Then, we explore the inherent relationships between adjacent complete labellings and extension-derived complete labellings.

\subsection{Syntax of $HAFS$s}
\begin{defn}
	A higher-order $AF$ with supports ($HAFS$) is a tuple 
	$(A, R, T,$ $U)$, where 
	$U=A \cup R \cup T$,
	$A$ is a finite set of arguments, 	
	$R\subseteq U \times U$ is a finite set of attacks,
	and $T\subseteq U \times U$ is a finite set of necessary supports.
	If $(a,b)\in R$, then we denote $(a,b)$ by $r_a^b$. If $(c,d)\in T$, then we denote $(c,d)$ by $t_c^d$.
\end{defn}
We say that $b$ attacks $a$ (or that $b$ is an attacker of $a$) iff $(b,a)\in R$, and we denote $(b, a)$ as $r_b^a$. We call that $c$ necessarily supports $a$ (or that $c$ is a necessary supporter of $a$) iff $(c,a)\in T$, and we denote $(c, a)$ as $t_c^a$.

\subsection{Adjacent complete labelling semantics}\label{labelling}
We first introduce the adjacent complete labelling of an $HAFS$.
\begin{defn}\label{adjacent}
	For an $HAFS=(A, R, T, U)$ and $\forall a\in U$, an adjacent complete labelling of the $HAFS$ is a total function $\|\cdot\|: U\to\{0,1,\frac{1}{2}\}$ such that
	\begin{itemize}
		\item 	$\|a\|=1$ iff [$\forall (b,a)\in R: \|b\|=0$ or $\|r_{b}^a\|=0$] and [$\forall (c,a)\in T: \|c\|=1$ or $\|t_{c}^a\|=0$];
		\item$\|a\|=0$ iff [$\exists (b,a)\in R: \|b\|=1$ and $\|r_{b}^a\|=1$] or [$\exists (c,a)\in T: \|c\|=0$ and $\|t_{c}^a\|=1$];
		\item $\|a\|= \frac{1}{2}$ iff otherwise.
	\end{itemize}
\end{defn}
Note that if $\nexists b((b,a)\in R)$ and $\nexists c((c,a)\in T)$ then the requirement for $\|a\|=1$ holds according to classical logic knowledge. We call the semantics ``adjacent complete labelling semantics'', since the value of an element in $U$ is decided only by its direct sources. Let us analyze the semantic feature characterized by this definition. For an element $a\in U$, if $a$ is validly attacked, i.e., it has a valid attacker and its associated valid attack, then $a$ is invalid; if $a$ is invalidly supported, i.e., it has an invalid supporter and its associated valid support, then $a$ is invalid; if $a$ is invalidly attacked and validly supported, then $a$ is valid; $a$ is undecided otherwise.

Next, we give the definition of the core and definitions of the grounded labelling, the preferred labelling and the stable labelling.
\begin{defn}
	For a labelling $L$ of an $HAFS=(A,R,T,U)$, the \emph{core} of $L$, denoted by $Core(L)$, is defined as the set $\{x\in U\mid L(x)=1\}$.
\end{defn}

\begin{defn}
	A labelling of an $HAFS=(A,R,T,U)$ is called
	\begin{itemize}
		 \item the \emph{adjacent grounded labelling} iff it is the adjacent complete labelling with the smallest (w.r.t. $\subseteq$) core; 
		\item an \emph{adjacent preferred labelling} iff it is an adjacent complete labelling with a maximal (w.r.t. $\subseteq$) core;  
		\item an \emph{adjacent stable labelling} $L_s$ iff $L_s$ is an adjacent preferred labelling such that $\forall y\in (U\setminus Core(L_s))$, $L_s(y)=0$.  
	\end{itemize}
\end{defn}
 
\subsection{Extension-based semantics for $HAFS$s}
We give the definitions of adjacent labelling semantics directly in Subsection \ref{labelling}. However, we want to establish traditional extension-based semantics for $HAFS$s and explore relationships with the adjacent complete labelling semantics.
\subsubsection{The extension-based semantic approach for $HAFS$s}
First, we present definitions related to support chains.
\begin{defn}
	Let $(A, R, T, U)$ be an $HAFS$.
	A \emph{support chain} of the $HAFS$ is a finite set $\{(f_1, f_2), (f_2, f_3), \dots, (f_{n-1}, f_n)\}\subseteq T$.
	An \emph{acyclic support chain} of the $HAFS$ is a support chain $\{(f_1, f_2), (f_2, f_3), \dots, (f_{n-1}, f_n)\}$ such that any two elements in $\{f_1, f_2, \dots, f_n\}$ are distinct.
	The $HAFS$ is \emph{support-acyclic} iff each support chain of the $HAFS$ is an acyclic support chain.
\end{defn}

Then, we introduce definitions related to defeats.
\begin{defn}\label{dft}
	Let $HAFS=(A, R, T, U)$, $B\subseteq U$, $a\in U$ and $b\in U$.
	We say that $b$ \emph{directly defeats} $a$ w.r.t. $B$, denoted as $b$ $\mathbf{a}$-$\mathbf{d}$ $a$ w.r.t. $B$, iff $b\in B$ and $(b,a)\in R\cap B$.
	We say that $b$ \emph{indirectly defeats} $a$ w.r.t. $B$ via $G$, denoted as $b$ $\mathbf{s}$-$\mathbf{d}$ $a$ w.r.t. $B$ via $G$, iff there exists an acyclic support chain $G=\{(g_n, g_{n-1}), (g_{n-1}, g_{n-2}), \dots, (g_2, g_1), (g_1, a)\}$ $(n\geqslant1)$ such that $G\subseteq T\cap B$, $(b, g_n)\in R\cap B$ and $b\in B$.
	We say that \emph{$b$ defeats $a$} w.r.t. $B$ iff $b$ $\mathbf{a}$-$\mathbf{d}$ $a$ w.r.t. $B$ or $b$ $\mathbf{s}$-$\mathbf{d}$ $a$ w.r.t. $B$.
	We say that \emph{$B$ defeats $a$} iff $\exists b\in B$ such that $b$ defeats $a$ w.r.t. $B$. 
	Let $\mathbf{dft}(B)$ denote the set $\{a\mid B \text{ defeats } a\}$.
\end{defn}
Note that the acyclic support chain is used for defining defeats but we do not require that the $HAFS$ in Definition \ref{dft} is support-acyclic.
\begin{lem}\label{lem1}
	Let $HAFS=(A, R, T, U)$, $B\subseteq U$, $a\in U$ and $b\in U$. $b$ $\mathbf{s}$-$\mathbf{d}$ $a$ w.r.t. $B$ via $G=\{(g_n, g_{n-1}), (g_{n-1}, g_{n-2}), \dots, (g_2, g_1), (g_1, a)\}$ $(n\geqslant1)$ iff $b$ defeats $g_1$ w.r.t. $B$ and $\{(g_1, a)\}\subseteq T\cap B$.
\end{lem}
\begin{proof}
	We first discuss two cases according to $n$.
	\begin{itemize}
		\item Case 1, $n=1$. \\
		$b$ $\mathbf{s}$-$\mathbf{d}$ $a$ w.r.t. $B$ via $G=\{(g_1, a)\}$
		\\$\Longleftrightarrow$ $b$ $\mathbf{a}$-$\mathbf{d}$ $g_1$ w.r.t. $B$ and $\{(g_1, a)\}\subseteq T\cap B$.
		\item Case 2, $n\geqslant 2$.\\
		$b$ $\mathbf{s}$-$\mathbf{d}$ $a$ w.r.t. $B$ via $G=\{(g_n, g_{n-1}), (g_{n-1}, g_{n-2}), \dots, (g_2, g_1), (g_1, a)\}$
		\\$\Longleftrightarrow$ there exists an acyclic support chain $$G=\{(g_n, g_{n-1}), (g_{n-1}, g_{n-2}), \dots, (g_2, g_1), (g_1, a)\}$$ $(n\geqslant1)$ such that $G\subseteq T\cap B$, $(b, g_n)\in R\cap B$ and $b\in B$
		\\$\Longleftrightarrow$ there exists an acyclic support chain $$G'=\{(g_n, g_{n-1}), (g_{n-1}, g_{n-2}), \dots, (g_2, g_1)\}$$ $(n\geqslant1)$ such that $G'\subseteq T\cap B$, $(b, g_n)\in R\cap B$, $b\in B$ and $\{(g_1, a)\}\subseteq T\cap B$
		\\$\Longleftrightarrow$ $b$ $\mathbf{s}$-$\mathbf{d}$ $g_1$ w.r.t. $B$ via $G'=\{(g_n, g_{n-1}), (g_{n-1}, g_{n-2}), \dots, (g_2, g_1)\}$ and $\{(g_1, a)\}\subseteq T\cap B$.
	\end{itemize}
	Then, by combining the above two cases, we have that:\\
	$b$ $\mathbf{s}$-$\mathbf{d}$ $a$ w.r.t. $B$ via $G=\{(g_n, g_{n-1}), (g_{n-1}, g_{n-2}), \dots, (g_2, g_1), (g_1, a)\}$ $(n\geqslant1)$ 
	\\$\Longleftrightarrow$ (1) (in the case of $n=1$) $b$ $\mathbf{a}$-$\mathbf{d}$ $g_1$ w.r.t. $B$ and $\{(g_1, a)\}\subseteq T\cap B$ or (2) (in the case of $n\geqslant 2$) $b$ $\mathbf{s}$-$\mathbf{d}$ $g_1$ w.r.t. $B$ via $G'=\{(g_n, g_{n-1}), (g_{n-1}, g_{n-2}), \dots, (g_2,$ $g_1)\}$ and $\{(g_1, a)\}\subseteq T\cap B$
	\\$\Longleftrightarrow$ $b$ defeats $g_1$ and $\{(g_1, a)\}\subseteq T\cap B$.
\end{proof}

Then we introduce the definition of defence.
\begin{defn}
	Let $HAFS=(A, R, T, U)$, $B\subseteq U$ and $a\in U$.
	We say that $B$ \emph{defends} $a$ iff each following item holds: 
	\begin{itemize}
		\item $\forall b((b, a)\in R)$: $B$ defeats $b$ or $B$ defeats $(b, a)$;
		\item $\forall c((c, a)\in T)$: $c\in B$ or $B$ defeats $(c, a)$.
	\end{itemize}
	Let $\mathbf{dfd}(B)$ denote the set $\{a\mid B \text{ defends } a\}$.
\end{defn}
Note that if $\nexists b((b,a)\in R)$ and $\nexists c((c,a)\in T)$ then the above two items hold according to classical logic knowledge.
Then we give the definition of related semantics based on defeats and defences. 
\begin{defn}
	Let $(A, R, T, U)$ be an $HAFS$, $E \subseteq U$ and $\forall a\in U$. We define that $E$ is:
	\begin{itemize}
		\item a \emph{conflict-free} set iff $\nexists a\in E$ such that $E$ defeats $a$;  
		\item an \emph{admissible} set iff it is a conflict-free set and $E\subseteq\mathbf{dfd}(E)$ (i.e. $\forall a\in E$: $E$ defends $a$);  
		\item \emph{a complete extension} iff it is a conflict-free set and $E=\mathbf{dfd}(E)$; 
		\item \emph{the grounded extension} iff it is the smallest (w.r.t. $\subseteq$) complete extension; 
		\item \emph{a preferred extension} iff it is a maximal (w.r.t. $\subseteq$) complete extension;  
		\item \emph{a stable extension} iff it is a preferred extension such that $E \cup \mathbf{dft}(E) = U$.  
	\end{itemize}
\end{defn}
Next, we give the definition of extension-derived complete labellings.
\begin{defn}
	For an $HAFS=(A, R, T, U)$, any complete extension $E$ of the $HAFS$ and $\forall a\in U$, an \emph{extension-derived complete labelling} of the $HAFS$ w.r.t. $E$ is a function $L_E: U\to\{0,1,\frac{1}{2}\}$ such that
	\begin{equation*}
		L_E(a)=\begin{cases}
			1 & \text{ iff } a\in E;\\
			0 & \text{ iff } a\in \mathbf{dft}(E);\\
			\frac{1}{2} & \text{ otherwise.}
		\end{cases}
	\end{equation*}
\end{defn}
Note that the complete extension $E$ is conflict-free. Thus, if $L_E(a)=1$ (i.e., $a\in E$) then $L_E(a)\neq 0$ (i.e. $a\notin \mathbf{dft}(E)$) and $L_E(a)\neq \frac{1}{2}$ (because of ``otherwise"). If $L_E(a)= 0$ (i.e. $a\in \mathbf{dft}(E)$) then $L_E(a)\neq1$ (i.e. $a\notin E$) and $L_E(a)\neq \frac{1}{2}$ (because of ``otherwise"). If $L_E(a)= \frac{1}{2}$ then $L_E(a)\neq1$ and $L_E(a)\neq 0$ (because of ``otherwise"). Therefore, an extension-derived complete labelling is a well-defined total function.

\subsubsection{Relationships between an extension-derived complete labelling and an adjacent complete labelling}
We present the relationship between an extension-derived complete labelling and an adjacent complete labelling of an $HAFS$.
\begin{thm}\label{thm1}
	For an $HAFS=(A, R, T, U)$, if a labelling is an extension-derived complete labelling of the $HAFS$ then it is an adjacent complete labelling of the $HAFS$.
\end{thm}
\begin{proof}
	Let $L_E$ be an extension-derived complete labelling of the $HAFS$ w.r.t. complete extension $E$. We need to discuss three cases for any $a\in U$.
	\begin{itemize}
		\item Case 1, $L_E(a)=0$.\\
		$L_E(a)=0$
		\\$\Longleftrightarrow$ $a\in \mathbf{dft}(E)$
		\\$\Longleftrightarrow$ $\exists e\in E$ such that $e$ defeats $a$ w.r.t. $E$
		\\$\Longleftrightarrow$ $\exists b\in E$ such that $b$ $\mathbf{a}$-$\mathbf{d}$ $a$ w.r.t. $E$ or $\exists c\in E$ such that $c$ $\mathbf{s}$-$\mathbf{d}$ $a$ w.r.t. $E$ via $G=\{(g_n, g_{n-1}), (g_{n-1}, g_{n-2}), \dots, (g_2, g_1), (g_1, a)\}$
		\\$\Longleftrightarrow$ $\exists b\in E$ such that $(b,a)\in R\cap E$ or by Lemma \ref{lem1}, $\exists c\in E$ such that $c$ defeats $g_1$ w.r.t. $E$ and $\{(g_1, a)\}\subseteq T\cap E$ 
		\\$\Longleftrightarrow$ by the definition of the extension-derived complete labelling, we have [$\exists (b,a)\in R: L_E(b)=1$ and $L_E(r_{b}^a)=1$] or [$\exists (g_1, a)\in T: L_E(g_1)=0$ and $L_E(t_{g_1}^a)=1$].
		\item Case 2, $L_E(a)=1$.\\
		$L_E(a)=1$
		\\$\Longleftrightarrow$ $a\in E$, i.e., $a\in \mathbf{dfd}(E)$
		\\$\Longleftrightarrow$ each following item holds: 
		\begin{itemize}
			\item $\forall b((b, a)\in R)$: $B$ defeats $b$ or $B$ defeats $(b, a)$
			\item $\forall c((c, a)\in T)$: $c\in B$ or $B$ defeats $(c, a)$
		\end{itemize}
		$\Longleftrightarrow$ by the definition of the extension-derived complete labelling, we have [$\forall (b,a)\in R: L_E(b)=0$ or $L_E(r_{b}^a)=0$] and [$\forall (c, a)\in T: L_E(c)=1$ or $L_E(t_{c}^a)=0$].
		\item Case 3, $L_E(a)=\frac{1}{2}$.\\
		$L_E(a)=\frac{1}{2}$
		\\$\Longleftrightarrow$ $L_E(a)\neq 0$ and $L_E(a)\neq 1$
		\\$\Longleftrightarrow$ from Case 1 and Case 2, it is not any case in the following items:
		\begin{itemize}
			\item in Case 1, [$\exists (b,a)\in R: L_E(b)=1$ and $L_E(r_{b}^a)=1$] or [$\exists (g_1, a)\in T: L_E(g_1)=0$ and $L_E(t_{g_1}^a)=1$]
			\item in Case 2 [$\forall (b,a)\in R: L_E(b)=0$ or $L_E(r_{b}^a)=0$] and [$\forall (c, a)\in T: L_E(c)=1$ or $L_E(t_{c}^a)=0$]
		\end{itemize}
	\end{itemize}
	From the three cases and Definition \ref{adjacent}, an extension-derived complete labelling $L_E$ is an adjacent complete labelling. 
\end{proof}
\begin{rmk}\label{rmk2}
	For an $HAFS=(A, R, T, U)$, an adjacent complete labelling of the $HAFS$ may not be an extension-derived complete labelling of the $HAFS$.
\end{rmk}
We use Example \ref{exmp3} to explain Remark \ref{rmk2}.
\begin{exmp}\label{exmp3}
	Let $HAFS=(A, R, T, U)$, $A=\{a\}$, $R=\emptyset$, $T=\{(a,a)\}$ and $U=\{a, (a,a)\}$. The extension-derived complete labelling $L_E$ of the $HAFS$ is that $L_E(a)=1$, and $L_E((a,a))=1$. However, three adjacent complete labellings of the $HAFS$, denoted $\|\cdot\|$, $\|\cdot\|'$, and $\|\cdot\|''$, are as follows: $[\|a\|=1$ and $\|(a,a)\|=1]$, $[\|a\|'=0$ and $\|(a,a)\|'=1]$, and $[\|a\|''=\frac{1}{2}$ and $\|(a,a)\|''=1]$.
\end{exmp} 

In fact, for some special cases, an adjacent complete labelling of an $HAFS$ is an extension-derived complete labelling of the $HAFS$.
\begin{thm}\label{thm2}
	For a support-acyclic $HAFS=(A, R, T, U)$, an adjacent complete labelling of the $HAFS$ is an extension-derived complete labelling of the $HAFS$.
\end{thm}
\begin{proof}
	Let $\|\cdot\|$ be an adjacent complete labelling and $E=\{a\in U\mid \|a\|=1\}$. We need to discuss three cases.
	
	First, we prove that $\forall d\in U$: $\|d\|=0$ iff $E$ defeats $d$. From Definition \ref{adjacent}, $\forall u\in U$, we have:
	\begin{equation}\label{0qui}		
		\| u \| = 0 \Longleftrightarrow [\exists  r_v^u \in E \text{ and } v\in E ] \text{ or } [ \exists t_w^u\in E \text{ and } \| w \| = 0]	
	\end{equation}
	Then, for any $a\in U$, we have the chain of equivalences below:\\
	$\|a\|=0$
	\\$\Longleftrightarrow$ [$\exists  r_b^a \in E$ and $b\in E$] or [$\exists t_c^a\in E$ and $\| c \| = 0$] by Equation \ref{0qui}
	\\$\Longleftrightarrow$ (Level 1) [$\exists b$: $b$ $\mathbf{a}$-$\mathbf{d}$ $a$ (w.r.t. $E$)] or [$\exists t_c^a\in E$ and $\| c \| = 0$]
	\\$\Longleftrightarrow$ [$\exists b$: $b$ $\mathbf{a}$-$\mathbf{d}$ $a$] or [$\exists t_c^a\in E$ and [[$\exists  r_d^c \in E$ and $d\in E$] or [$\exists t_e^c\in E$ and $\| e \| = 0$]]] by Equation \ref{0qui}
	\\$\Longleftrightarrow$ [$\exists b$: $b$ $\mathbf{a}$-$\mathbf{d}$ $a$] or [$\exists t_c^a\in E$ and $\exists  r_d^c \in E$ and $d\in E$] or [$\exists t_c^a\in E$ and $\exists t_e^c\in E$ and $\| e \| = 0$]
	\\$\Longleftrightarrow$ (Level 2) [$\exists b$: $b$ $\mathbf{a}$-$\mathbf{d}$ $a$] or [$\exists d$: $d$ $\mathbf{s}$-$\mathbf{d}$ $a$ (w.r.t. $E$ via $\{t_c^a\}$)] or [$\exists t_c^a\in E$ and $\exists t_e^c\in E$ and $\| e \| = 0$]
	\\$\Longleftrightarrow$ [$\exists b$: $b$ $\mathbf{a}$-$\mathbf{d}$ $a$] or [$\exists d$: $d$ $\mathbf{s}$-$\mathbf{d}$ $a$] or [$\exists t_c^a\in E$ and $\exists t_e^c\in E$ and [[$\exists  r_f^e \in E$ and $f\in E$] or [$\exists t_g^e\in E$ and $\| g \| = 0$]]] by Equation \ref{0qui}
	\\$\Longleftrightarrow$ [$\exists b$: $b$ $\mathbf{a}$-$\mathbf{d}$ $a$] or [$\exists d$: $d$ $\mathbf{s}$-$\mathbf{d}$ $a$] or [$\exists t_c^a\in E$ and $\exists t_e^c\in E$ and $\exists  r_f^e \in E$ and $f\in E$] or [$\exists t_c^a\in E$ and $\exists t_e^c\in E$ and $\exists t_g^e\in E$ and $\| g \| = 0$]
	\\$\Longleftrightarrow$ (Level 3) [$\exists b$: $b$ $\mathbf{a}$-$\mathbf{d}$ $a$] or [$\exists d$: $d$ $\mathbf{s}$-$\mathbf{d}$ $a$] or [$\exists f$: $f$ $\mathbf{s}$-$\mathbf{d}$ $a$ (w.r.t. $E$ via $\{t_c^a, t_e^c\}$)] or [$\exists t_c^a\in E$ and $\exists t_e^c\in E$ and $\exists t_g^e\in E$ and $\| g \| = 0$]
	
	Since the $HAFS$ is support-acyclic and $U$ is finite, this chain of equivalences must stop at a finite level and we suppose that it stops at Level n.
	Thus, we have:\\
	$\|a\|=0$
	\\$\Longleftrightarrow$ (Level n) [$\exists b$: $b$ $\mathbf{a}$-$\mathbf{d}$ $a$] or [$\exists d$: $d$ $\mathbf{s}$-$\mathbf{d}$ $a$] or [$\exists f$: $f$ $\mathbf{s}$-$\mathbf{d}$ $a$] or $\dots$ or [$\exists t_c^a\in E$ and $\exists t_e^c\in E$ and $\exists t_g^e\in E$ and $\dots$ and $\exists t_k^j\in E$ and $\| k \| = 0$]
	\\$\Longleftrightarrow$ [$\exists b$: $b$ $\mathbf{a}$-$\mathbf{d}$ $a$] or [$\exists d$: $d$ $\mathbf{s}$-$\mathbf{d}$ $a$] or [$\exists f$: $f$ $\mathbf{s}$-$\mathbf{d}$ $a$] or $\dots$ or [$\exists t_c^a\in E$ and $\exists t_e^c\in E$ and $\exists t_g^e\in E$ and $\dots$ and $\exists t_k^j\in E$ and $\exists  r_m^k \in E$ and $m\in E$]
	by Equation \ref{0qui} and the final Level n
	\\$\Longleftrightarrow$ [$\exists b$: $b$ $\mathbf{a}$-$\mathbf{d}$ $a$] or [$\exists d$: $d$ $\mathbf{s}$-$\mathbf{d}$ $a$] or [$\exists f$: $f$ $\mathbf{s}$-$\mathbf{d}$ $a$] or $\dots$ or [$\exists m$: $m$ $\mathbf{s}$-$\mathbf{d}$ $a$]
	\\$\Longleftrightarrow$ $E$ defeats $a$.
	
	Second, we prove that $E$ is a complete extension.\\
	$\|a\|=1$ (i.e., $a\in E$)
	\\$\Longleftrightarrow$ [$\forall (b,a)\in R: \|b\|=0$ or $\|r_{b}^a\|=0$] and [$\forall (c,a)\in T: \|c\|=1$ or $\|t_{c}^a\|=0$] by Definition \ref{adjacent}
	\\$\Longleftrightarrow$ [$\forall (b,a)\in R:$ $E$ defeats $b$ or $E$ defeats $r_{b}^a$] and [$\forall (c,a)\in T: c\in E$ or $E$ defeats $t_{c}^a$] 
	\\$\Longleftrightarrow$ $a\in \mathbf{dfd}(E)$.
	
	Thus, we obtain that $E= \mathbf{dfd}(E)$. We also need to prove that $E$ is conflict-free.
	\begin{itemize}
		\item $a\in E (\|a\|=1)$
		\\$\Longrightarrow$ $\forall (b,a)\in R: \|b\|=0$ or $\|r_{b}^a\|=0$
		\\$\Longrightarrow$ $\nexists (b,a)\in R: \|b\|=1$ and $\|r_{b}^a\|=1$
		\\$\Longleftrightarrow$ $\nexists b: b$ $\mathbf{a}$-$\mathbf{d}$ $a$ w.r.t. $E$.
		\item $a\in E (\|a\|=1)$
		\\$\Longrightarrow$ $\forall (c,a)\in T: \|c\|=1$ or $\|t_{c}^a\|=0$
		\\$\Longrightarrow$ $\nexists (c,a)\in T: \|c\|=0$ and $\|t_{c}^a\|=1$
		\\$\Longleftrightarrow$ $\nexists (c,a)\in T:$ [$\exists e: e$ defeats $c$ w.r.t. $E$] and [$t_{c}^a\in E$]		
		\\$\Longleftrightarrow$ $\nexists (c,a)\in T:$ [$\exists e: e$ $\mathbf{s}$-$\mathbf{d}$ $a$ w.r.t. $E$ via $\{\dots,(c,a)\}$] by Lemma \ref{lem1}
		\\$\Longleftrightarrow$ $\nexists g: g$ $\mathbf{s}$-$\mathbf{d}$ $a$ w.r.t. $E$
	\end{itemize}
	From above two items, we have that $\forall a\in E$: $E$ does not defeat $a$, i.e., $E$ is conflict-free. Thus we have proven that $E$ is a complete extension.
	
	Third, we analyse the case of $\|a\|=\frac{1}{2}$.\\
	$\|a\|=\frac{1}{2}$
	\\$\Longleftrightarrow$ $\|a\|\neq1$ and $\|a\|\neq0$
	\\$\Longleftrightarrow$ $a\notin E$ and $E$ does not defeat $a$.
	
	From the three cases, an adjacent complete labelling $\|\cdot\|$ of a support-acyclic $HAFS$ satisfies the definition of an extension-derived complete labelling of the $HAFS$.
\end{proof}
A $BHAF$ \cite{barringer2005temporal, tang2025encodingarg} is a particular $HAFS$ whose support relation is the empty set. Thus, we have established the extension-based semantics for $BHAF$s by applying the extension-based semantics for $HAFS$s to $BHAF$s. And an adjacent complete labelling of a $BHAF$ is equivalent to an extension-derived complete labelling of the $BHAF$ by Theorem \ref{thm1} and Theorem \ref{thm2}.

\section{Encoding $HAFS$s to $\mathcal{PLS}$s}
This section is dedicated to encoding $HAFS$s to $\mathcal{PLS}$s, which serves to link $HAFS$s with $\mathcal{PLS}$s and dissect their semantic correlations. 

We first present a general encoding approach for $HAFS$s. Subsequently, we elaborate on specific encoding strategies: we map $HAFS$s to $\mathcal{PL}_{3}^L$, encode $HAFS$s to $\mathcal{PL}_{[0,1]}$ (with a focus on the fuzzy normal encoded semantics), and further realize the encoding of $HAFS$s to three important instances of $\mathcal{PL}_{[0,1]}$. Finally, we investigate the relationships between the adjacent complete labelling semantics of $HAFS$s and the fuzzy normal encoded semantics, which are analyzed in three scenarios: zero-divisor free t-norms, $\frac{1}{2}$-idempotent t-norms, and zero-divisor free $\frac{1}{2}$-idempotent t-norms.
\subsection{The encoding approach for $HAFS$s}
In this subsection, we present related definitions of the encoding methodology and the translating methodology for $HAFS$s. Similar approaches are introduced in \cite{tang2025encoding, tang2025encodingarg} .

Let $\mathcal{HAFS}$ be the set of all $HAFS$s. Let the arguments, attacks and supports of $\mathcal{HAFS}$ be in one-to-one correspondence with propositional variables in a $\mathcal{PLS}$. Let $\mathcal{F_{\mathcal{PL}}}$ be the set of all formulas in the $\mathcal{PLS}$. 

\begin{defn}
	An \emph{encoding} of $\mathcal{HAFS}$ is a function $ec: \mathcal{HAFS} \rightarrow \mathcal{F_{\mathcal{PL}}}$, s.t. the set of all propositional variables in the formula $ec(HAFS)$ is the set of all arguments, attacks and supports in the $HAFS$. We call $ec(HAFS)$ the \emph{encoded formula} of the $HAFS$.
\end{defn}
Encoding $HAFS$s to a $\mathcal{PLS}$ refers to the process of mapping them to $\mathcal{F_{\mathcal{PL}}}$ via a given encoding function. 
\begin{defn}\label{defn13}
	For a given encoding function $ec$ and the associated $\mathcal{PLS}$, the \emph{encoded argumentation semantics} (encoded semantics, briefly) of $\mathcal{HAFS}$ is the semantics obtained by defining models of each $HAFS$ as models of the $ec(HAFS)$ in the associated $\mathcal{PLS}$.
	A \emph{fuzzy encoded argumentation semantics} (fuzzy encoded semantics, briefly) of $\mathcal{HAFS}$ is an encoded argumentation semantics where the $\mathcal{PLS}$ is instantiated as a $\mathcal{PL}_{[0,1]}$.
\end{defn}

\begin{defn}
	A \emph{translation} of $\mathcal{HAFS}$ under a given semantics is an encoding function $tr: \mathcal{HAFS} \rightarrow \mathcal{F_{\mathcal{PL}}}$, $HAFS \mapsto tr(HAFS)$, s.t. models of $tr(HAFS)$ in the $\mathcal{PLS}$ are in one-to-one correspondence with the models of the $HAFS$ under the given semantics. 
\end{defn}

\begin{defn}
	A given semantics of  $\mathcal{HAFS}$ is called \emph{translatable} if there exists a translation of $\mathcal{HAFS}$.
\end{defn} 
Obviously, an encoded argumentation semantics of $\mathcal{HAFS}$ is translatable. 

Within a $\mathcal{PLS}$, the labelling for an $HAFS$ is conceptually parallel to the assignment for its encoded formula. Thus, for the sake of ease of expression, we consider the terms ``assignment'' and ``labelling'' interchangeable. 

In this paper, we highlight a particular encoding of $HAFS$s, named a normal encoding of $HAFS$s (w.r.t a $\mathcal{PLS}$). 

\begin{defn}
	The normal encoding of $HAFS$s (w.r.t a $\mathcal{PLS}$) is a function $ec_{HN}: \mathcal{HN} \to \mathcal{F_{PL}}$ such that for a given $HAFS=(A, R, T, U)$ we have that
	\begin{equation*}
		ec_{HN}(HAFS)=\bigwedge_{a\in U}(a\leftrightarrow(\bigwedge_{(b, a)\in R}\neg(r_{b}^{a}\wedge b))\wedge(\bigwedge_{(c, a)\in T}\neg(t_{c}^{a}\wedge \neg c))).
	\end{equation*}
\end{defn}
\subsection{Encoding $HAFS$s to $\mathcal{PL}_3^L$}
Let the $\mathcal{PLS}$ in the definition of the normal encoding be $\mathcal{PL}_3^L$. We present the semantic equivalence theorem between adjacent complete labelling semantics and encoded semantics associated with $\mathcal{PL}_3^L$ as follows. For brevity, in the remainder of this subsection, the term complete semantics shall refer to the adjacent complete labelling semantics.
\begin{thm}
	For an $HAFS=(A, R, T, U)$, an assignment is a model of the $HAFS$ under complete semantics iff it is a model of $ec_{HN}(HAFS)$ in $\mathcal{PL}_3^L$.
\end{thm}
\begin{proof} 
	We need to check that for a given assignment $\|\cdot\|$ and for each $a\in U$, any value of $a$ satisfies complete semantics iff $\|\bigwedge_{a\in U}(a\leftrightarrow(\bigwedge_{(b, a)\in R}\neg(r_{b}^{a}\wedge b))\wedge(\bigwedge_{(c, a)\in T}\neg(t_{c}^{a}\wedge \neg c)))\|=1$ in $\mathcal{PL}_3^L$. We need to discuss three cases.
	\begin{itemize}
		\item Case 1, $\|a\|=1$. 
		\\$\|a\|=1$ by model $\|\cdot\|$ under complete semantics
		\\$\Longleftrightarrow$ $[\forall (b,a)\in R: \|b\|=0$ or $\|r_{b}^a\|=0]$ and $[\forall (c,a)\in T: \|c\|=1$ or $\|t_{c}^a\|=0]$  by complete semantics
		\\$\Longleftrightarrow$ $[\forall (b,a)\in R: \|r_{b}^{a}\wedge b\|=0]$ and $[\forall (c,a)\in T: \|t_{c}^{a}\wedge\neg c\|=0]$ in $\mathcal{PL}_3^L$		
		\\$\Longleftrightarrow$ $[\forall (b,a)\in R: \|\neg(r_{b}^{a}\wedge b)\|=1]$ and $[\forall (c,a)\in T: \|\neg(t_{c}^{a}\wedge\neg c)\|=1]$ in $\mathcal{PL}_3^L$
		\\$\Longleftrightarrow$ $\|\bigwedge_{(b, a)\in R}\neg(r_{b}^{a}\wedge b)\|=1$ and $\|\bigwedge_{(c, a)\in T}\neg(t_{c}^{a}\wedge\neg c)\|=1$ in $\mathcal{PL}_3^L$		
		\\$\Longleftrightarrow$ $\|\bigwedge_{(b, a)\in R}\neg(r_{b}^{a}\wedge b)\wedge\bigwedge_{(c, a)\in T}\neg(t_{c}^{a}\wedge\neg c)\|=1$ in $\mathcal{PL}_3^L$
		\\$\Longleftrightarrow$ $\|a\leftrightarrow(\bigwedge_{(b, a)\in R}\neg(r_{b}^{a}\wedge b))\wedge(\bigwedge_{(c, a)\in T}\neg(t_{c}^{a}\wedge\neg c))\|=1$ in $\mathcal{PL}_3^L$.
		
		\item Case 2, $\|a\|=0$. 
		\\$\|a\|=0$ by model $\|\cdot\|$ under complete semantics
		\\$\Longleftrightarrow$ [$\exists (b,a)\in R: \|b\|=1$ and $\|r_{b}^a\|=1$] or [$\exists (c,a)\in T: \|c\|=0$ and $\|t_{c}^a\|=1$] by complete semantics
		\\$\Longleftrightarrow$ [$\exists (b,a)\in R: \|r_{b}^{a}\wedge b\|=1$] or [$\exists (c,a)\in T: \|t_{c}^{a}\wedge\neg c\|=1$] in $\mathcal{PL}_3^L$		
		\\$\Longleftrightarrow$ [$\exists (b,a)\in R: \|\neg(r_{b}^{a}\wedge b)\|=0$] or [$\exists (c,a)\in T: \|\neg(t_{c}^{a}\wedge\neg c)\|=0$] in $\mathcal{PL}_3^L$
		\\$\Longleftrightarrow$ $\|\bigwedge_{(b, a)\in R}\neg(r_{b}^{a}\wedge b)\|=0$ or $\|\bigwedge_{(c, a)\in T}\neg(t_{c}^{a}\wedge\neg c)\|=0$ in $\mathcal{PL}_3^L$		
		\\$\Longleftrightarrow$ $\|\bigwedge_{(b, a)\in R}\neg(r_{b}^{a}\wedge b)\wedge\bigwedge_{(c, a)\in T}\neg(t_{c}^{a}\wedge\neg c)\|=0$ in $\mathcal{PL}_3^L$
		\\$\Longleftrightarrow$ $\|a\leftrightarrow(\bigwedge_{(b, a)\in R}\neg(r_{b}^{a}\wedge b))\wedge(\bigwedge_{(c, a)\in T}\neg(t_{c}^{a}\wedge\neg c))\|=1$ in $\mathcal{PL}_3^L$.
		
		\item Case 3, $\|a\|=\frac{1}{2}$. 
		\\$\|a\|=\frac{1}{2}$ by model $\|\cdot\|$ under complete semantics
		\\$\Longleftrightarrow$ $\|\beta\|\neq1$ and $\|\beta\|\neq0$ under complete semantics
		\\$\Longleftrightarrow$ $\|\bigwedge_{(b, a)\in R}\neg(r_{b}^{a}\wedge b)\wedge\bigwedge_{(c, a)\in T}\neg(t_{c}^{a}\wedge\neg c)\|\neq 1$ and $\|\bigwedge_{(b, a)\in R}\neg(r_{b}^{a}\wedge b)\wedge\bigwedge_{(c, a)\in T}\neg(t_{c}^{a}\wedge\neg c)\|\neq 0$ in $\mathcal{PL}_3^L$ by Case 1 and Case 2
		\\$\Longleftrightarrow$ $\|\bigwedge_{(b, a)\in R}\neg(r_{b}^{a}\wedge b)\wedge\bigwedge_{(c, a)\in T}\neg(t_{c}^{a}\wedge\neg c)\|=\frac{1}{2}$ in $\mathcal{PL}_3^L$.
		\\$\Longleftrightarrow$ $\|a\leftrightarrow(\bigwedge_{(b, a)\in R}\neg(r_{b}^{a}\wedge b))\wedge(\bigwedge_{(c, a)\in T}\neg(t_{c}^{a}\wedge\neg c))\|=1$ in $\mathcal{PL}_3^L$.
	\end{itemize}
	From the three cases above, for a given assignment $\|\cdot\|$, the value $\|a\|$ of any $a\in U$ satisfies complete semantics iff $\|a\leftrightarrow(\bigwedge_{(b, a)\in R}\neg(r_{b}^{a}\wedge b))\wedge(\bigwedge_{(c, a)\in T}\neg(t_{c}^{a}\wedge\neg c))\|=1$ in $\mathcal{PL}_3^L$. 
	Thus, an assignment $\|\cdot\|$ of the $HAFS$ satisfies complete semantics iff the assignment $\|\cdot\|$ is a model of the $ec_{HN}(HAFS)$ in $\mathcal{PL}_3^L$.
\end{proof}

\subsection{Encoding $HAFS$s to $\mathcal{PL}_{[0,1]}$s: fuzzy normal encoded semantics}
We first present the equational systems derived from fuzzy normal encoded semantics. Assume a $\mathcal{PL}_{[0,1]}$ equipped with connectives $\neg$, $\wedge$, and $\rightarrow$, where the three connectives are interpreted as the negation operation $N$, the t-norm operation $\ast$ and the R-implication operation $I_\ast$, respectively.

\begin{thm}\label{thm7}
	For an $HAFS=(A, R, T, U)$, an assignment $\|\cdot\|$ is a model of $ec_{HN}(HAFS)$ in the $\mathcal{PL}_{[0,1]}$ iff it is a solution of the equational system $Eq^{ec_{HN}}$:  $\forall a\in U,$
	\begin{equation}\label{eq-ec}
	\|a\| =\mathbf{N_b}\ast \mathbf{N_c},
\end{equation}
where $\mathbf{N_b}=N(\|b_1\|\ast\|r_{b_1}^a\|)\ast\dots\ast N(\|b_k\|\ast\|r_{b_k}^a\|)$, $\mathbf{N_c}=N(N(\|c_1\|)\ast\|t_{c_1}^a\|)\ast\dots\ast N(N(\|c_m\|)\ast\|t_{c_m}^a\|)$, $\{b_1, \dots , b_k \}$ is the set of all attackers of $a$, and $\{c_1, \dots , c_m \}$ is the set of all (necessary) supporters of $a$. 
\end{thm}
\begin{proof}
	An assignment $\|\cdot\|$ is a model of $ec_{HN}(HAFS)$ in the $\mathcal{PL}_{[0,1]}$\\
	$\Longleftrightarrow$
	\begin{equation*}
		\|\bigwedge_{a\in U}(a\leftrightarrow(\bigwedge_{(b_i, a)\in R}\neg(r_{b_i}^{a}\wedge b_i))\wedge(\bigwedge_{(c_j, a)\in T}\neg(\neg t_{c_j}^{a}\wedge c_j)))\|=1,
	\end{equation*}
	$\Longleftrightarrow$ for each $a \in U$, 
	\begin{equation*}
		\|a\leftrightarrow(\bigwedge_{(b_i, a)\in R}\neg(r_{b_i}^{a}\wedge b_i))\wedge(\bigwedge_{(c_j, a)\in T}\neg(\neg t_{c_j}^{a}\wedge c_j))\|=1,
	\end{equation*}
	$\Longleftrightarrow$ for each $a \in U$, 
	\begin{equation*}
		\|a\|=\|(\bigwedge_{(b_i, a)\in R}\neg(r_{b_i}^{a}\wedge b_i))\wedge(\bigwedge_{(c_j, a)\in T}\neg(\neg t_{c_j}^{a}\wedge c_j))\|,
	\end{equation*}
	i.e.,
	\begin{equation*}
			\|a\| =	\mathbf{N_b}\ast \mathbf{N_c}, 
		\end{equation*}
		where $\mathbf{N_b}=N(\|b_1\|\ast\|r_{b_1}^a\|)\ast\dots\ast N(\|b_k\|\ast\|r_{b_k}^a\|)$ and $\mathbf{N_c}=N(N(\|c_1\|)\ast\|t_{c_1}^a\|)\ast\dots\ast N(N(\|c_m\|)\ast\|t_{c_m}^a\|)$ and $\{b_1, \dots , b_k \}$ is the set of all attackers of $a$ and $\{c_1, \dots , c_m \}$ is the set of all (necessary) supporters of $a$. 
\end{proof}

\begin{defn}\label{defn29}
	For any $HAFS$, the equational system $Eq^{ec_{HN}}$ associated with a fuzzy normal encoded semantics is called the \emph{fuzzy normal encoded equational system} (abbreviated as the \emph{encoded equational system}).
\end{defn}
Without causing confusion, and based on Theorem \ref {thm7}, we also denote the fuzzy normal encoded semantics associated with the given $\mathcal{PL}_{[0,1]}$ by $Eq^{ec_{HN}}$.
If $N$ and $\ast$ are continuous, then the Equation \ref{eq-ec} is continuous, and we denote the associated continuous fuzzy normal encoded equational system as $Eq_{con}^{ec_{HN}}$. Then from the existence theorem of the solution of a continuous equational system \cite{tang2025encodingarg}, we have the existence theorem of the solution of $Eq_{con}^{ec_{HN}}$ below.
\begin{thm}
	There exists a solution of the continuous fuzzy normal encoded equational system $Eq_{con}^{ec_{HN}}$ for a given $HAFS$.
\end{thm}

\begin{defn}
	The \emph{encoded equational function} associated with $Eq^{ec_{HN}}$ is defined as a function $h^{ec_{HN}}: [0, 1]^{2k+2m} \to [0, 1]$ ($k\geqslant 0$, $m\geqslant 0$),
	\begin{equation*}
	 h^{ec_{HN}}(x_1, x_1', \dots, x_k, x_k', y_1, y_1', \dots, y_m, y_m')=\mathbf{N_x}\ast \mathbf{N_y},
	\end{equation*}	
	 where $\mathbf{N_x}=N(x_1\ast x_1')\ast\dots\ast N(x_k\ast x_k')$ and $\mathbf{N_y}=N(N(y_1)\ast y_1')\ast\dots\ast N(N(y_m)\ast y_m')$. 
\end{defn}
Let $Dec=\{x_1, x_1', \dots, x_k, x_k', y_1', \dots, y_m'\}$ and $Inc=\{y_1, \dots, y_m\}$. It is easy to see that the equational function $h^{ec_{HN}}$ satisfies decreasing monotonicity w.r.t. any variable in $Dec$ and satisfies increasing monotonicity w.r.t. any variable in $Inc$.

Next we state and demonstrate some properties of the encoded equational function $h^{ec_{HN}}$.

\begin{thm}\label{}
	The encoded equational function $h^{ec_{HN}}$ satisfies the boundary conditions and the symmetry:\\
    (a) $h^{ec_{HN}}(x_1, x_1', \dots, x_k, x_k', y_1, y_1', \dots, y_m, y_m') = 1$, if $[\forall x_i\in \{x_1, \dots, x_k\}: \|x_i\|=0 \text{ or } \|x_i'\|=0] \text{ and } [\forall y_j\in \{y_1, \dots, y_m\}: \|y_j\|=1 \text{ or } \|y_j'\|=0]$\\
	(b) $h^{ec_{HN}}(x_1, x_1', \dots, x_k, x_k', y_1, y_1', \dots, y_m, y_m') = 0$, if $[\exists x_i\in \{x_1, \dots, x_k\}: \|x_i\|=1 \text{ and } \|x_i'\|=1] \text{ or } [\exists  y_j\in \{y_1, \dots, y_m\}: \|y_j\|=0 \text{ and } \|y_j'\|=1]$.\\
	(c) \begin{align*}
		&h^{ec_{HN}}(x_1, x_1', \dots, x_k, x_k', y_1, y_1', \dots, y_m, y_m')\\ 
		=& h^{ec_{HN}}(x_{\sigma(1)}, x_{\sigma(1)}', \dots, x_{\sigma(k)}, x_{\sigma(k)}', y_{\tau(1)}, y_{\tau(1)}', \dots, y_{\tau(m)}, y_{\tau(m)}')
	\end{align*} for any permutation $\sigma$ of $\{1, \dots, k\}$ and any permutation $\tau$ of $\{1, \dots, m\}$.
\end{thm}
\begin{proof}
	(a) $[\forall x_i\in \{x_1, \dots, x_k\}: \|x_i\|=0 \text{ or } \|x_i'\|=0] \text{ and } [\forall y_j\in \{y_1, \dots,$ $y_m\}: \|y_j\|=1 \text{ or } \|y_j'\|=0]$
	\\$\Longrightarrow$ $x_1\ast x_1'=\dots= x_k\ast x_k'= N(y_1)\ast y_1'=\dots= N(y_m)\ast y_m'=0$
	\\$\Longleftrightarrow$ $N(x_1\ast x_1')=\dots= N(x_k\ast x_k')= N(N(y_1)\ast y_1')=\dots= N(N(y_m)\ast y_m')=1$
	\\$\Longleftrightarrow$ $N(x_1\ast x_1')\ast\dots\ast N(x_k\ast x_k')\ast N(N(y_1)\ast y_1')\ast\dots\ast N(N(y_m)\ast y_m')=1$
	\\$\Longleftrightarrow$ $h^{ec_{HN}}(x_1, x_1', \dots, x_k, x_k', y_1, y_1', \dots, y_m, y_m') = 1$.
	\\(b) $\exists x_i\in \{x_1, \dots, x_k\}: \|x_i\|=1$ and $\|x_i'\|=1$ or $\exists  y_j\in \{y_1, \dots, y_m\}: \|y_j\|=0$ and $\|y_j'\|=1$ 
	\\$\Longleftrightarrow$ $\exists x_i\in \{x_1, \dots, x_k\}: x_i\ast x_i'=1$ or $\exists  y_j\in \{y_1, \dots, y_m\}: N(y_j)\ast y_j'=1$
	\\$\Longleftrightarrow$ $\exists x_i\in \{x_1, \dots, x_k\}: N(x_i\ast x_i')=0$ or $\exists  y_j\in \{y_1, \dots, y_m\}: N(N(y_j)\ast y_j')=0$
	\\$\Longleftrightarrow$ $h^{ec_{HN}}(x_1, x_1', \dots, x_k, x_k', y_1, y_1', \dots, y_m, y_m') = 0$
	\\(c) By the commutativity and the associativity of any t-norm.
\end{proof}
Obviously, the encoded equational function $h^{ec_{HN}}$ may not be continuous since $N$ and $\ast$ may not be continuous. However, if $N$ and $\ast$ are continuous, then the encoded equational function $h^{ec_{HN}}$ will be continuous.

\subsection{Encoding $HAFS$s to three important $\mathcal{PL}_{[0, 1]}$s}
In this subsection, we will encode $AF$s to three specific $\mathcal{PL}_{[0, 1]}s$ ($\mathcal{PL}_{[0, 1]}^G$, $\mathcal{PL}_{[0, 1]}^P$ and $\mathcal{PL}_{[0, 1]}^L$).

We first present the equational system $Eq^G$:  $\forall a\in U,$
\begin{equation*}
	\|a\| =\min\{\min_{i=1}^k\max\{1-\|b_i\|, 1-\|r_{b_i}^a\|\}, \min_{i=1}^m\max\{\|c_i\|, 1-\|t_{c_i}^a\|\}\},
\end{equation*}
where $\{b_1, \dots , b_k \}$ is the set of all attackers of $a$, and $\{c_1, \dots , c_m \}$ is the set of all (necessary) supporters of $a$. 
\begin{thm}
	An assignment is a model of an $HAFS$ under equational semantics $Eq^G$ iff it is a model of $ec_{HN}(HAFS)$ in $\mathcal{PL}_{[0, 1]}^G$. 
\end{thm} 
\begin{proof}
	Let the $\mathcal{PL}_{[0, 1]}$ in Theorem \ref{thm7} be $\mathcal{PL}_{[0, 1]}^G$ with the standard negation, the G\"{o}del t-norm $\ast$, and the R-implication $I_\ast$. Then from Theorem \ref{thm7} for an $HAFS=(A, R, T, U)$, an assignment $\|\cdot\|$ is a model of $ec_{HN}(HAFS)$ in $\mathcal{PL}_{[0, 1]}^G$ iff it is a solution of the equational system: $\forall a\in U$,
	\begin{equation*}
		\|a\| =\mathbf{N_b^G}\ast \mathbf{N_c^G},
	\end{equation*}
	where $\{b_1, \dots , b_k \}$ is the set of all attackers of $a$, $\{c_1, \dots , c_m \}$ is the set of all (necessary) supporters of $a$, and we have that
	\begin{align*}
		\mathbf{N_b^G}
		&=N(\|b_1\|\ast\|r_{b_1}^a\|)\ast\dots\ast N(\|b_k\|\ast\|r_{b_k}^a\|)\\
		&=\min_{i=1}^k(1-\min\{\|b_i\|, \|r_{b_i}^a\|\})\\
		&=\min_{i=1}^k\max\{1-\|b_i\|, 1-\|r_{b_i}^a\|\}\\
	\end{align*}
	and
	\begin{align*}
		\mathbf{N_c^G}
		&=N(N(\|c_1\|)\ast\|t_{c_1}^a\|)\ast\dots\ast N(N(\|c_m\|)\ast\|t_{c_m}^a\|)\\
		&=\min_{i=1}^m(1-\min\{1-\|c_i\|, \|t_{c_i}^a\|\})\\
		&=\min_{i=1}^m\max\{\|c_i\|, 1-\|t_{c_i}^a\|\}.\\
	\end{align*}
Thus we have that
\begin{align*}
	\mathbf{N_b^G}\ast \mathbf{N_c^G}=\min\{\min_{i=1}^k\max\{1-\|b_i\|, 1-\|r_{b_i}^a\|\}, \min_{i=1}^m\max\{\|c_i\|, 1-\|t_{c_i}^a\|\}\}.
\end{align*}

	Thus, an assignment $\|\cdot\|$ is a model of $ec_{HN}(HAFS)$ in $\mathcal{PL}_{[0, 1]}^G$
	 \\$\Longleftrightarrow$ it is a solution of the equational system: $\forall a\in U$,
\begin{equation*}
	\|a\|=\min\{\min_{i=1}^k\max\{1-\|b_i\|, 1-\|r_{b_i}^a\|\}, \min_{i=1}^m\max\{\|c_i\|, 1-\|t_{c_i}^a\|\}\},
\end{equation*}
where $\{b_1, \dots , b_k \}$ is the set of all attackers of $a$, and $\{c_1, \dots , c_m \}$ is the set of all (necessary) supporters of $a$
	\\$\Longleftrightarrow$ it is a model of the $HAFS$ under equational semantics $Eq^G$.
\end{proof}

From this theorem, we immediately obtain the following corollary.
\begin{cor}
	The equational semantics $Eq^G$ is translatable via $ec_{HN}$ and $\mathcal{PL}_{[0, 1]}^G$, i.e., it is equivalent to a fuzzy normal encoded semantics associated with $\mathcal{PL}_{[0, 1]}^G$.
\end{cor}

Next, we present the equational semantics $Eq^P$ and show the theorem of model equivalence between $Eq^P$ and the fuzzy normal encoded semantics associated with $\mathcal{PL}_{[0, 1]}^P$.

We give the equational system $Eq^P$:  $\forall a\in U,$
\begin{equation*}
	\|a\|=\prod_{i=1}^k(1-\|b_i\|\|r_{b_i}^a\|)\prod_{i=1}^m(1-\|t_{c_i}^a\|+\|c_i\|\|t_{c_i}^a\|),
\end{equation*}
where $\{b_1, \dots , b_k \}$ is the set of all attackers of $a$, and $\{c_1, \dots , c_m \}$ is the set of all (necessary) supporters of $a$. 
\begin{thm}
	An assignment is a model of an $HAFS$ under equational semantics $Eq^P$ iff it is a model of $ec_{HN}(HAFS)$ in $\mathcal{PL}_{[0, 1]}^P$. 
\end{thm} 
\begin{proof}
	Let the $\mathcal{PL}_{[0, 1]}$ in Theorem \ref{thm7} be $\mathcal{PL}_{[0, 1]}^P$ with the standard negation, the Product t-norm $\ast$, and the R-implication $I_\ast$. Then from Theorem \ref{thm7} for an $HAFS=(A, R, T, U)$, an assignment $\|\cdot\|$ is a model of $ec_{HN}(HAFS)$ in $\mathcal{PL}_{[0, 1]}^P$ iff it is a solution of the equational system: $\forall a\in U$,
	\begin{equation*}
		\|a\| =\mathbf{N_b^P}\ast \mathbf{N_c^P},
	\end{equation*}
	where $\{b_1, \dots , b_k \}$ is the set of all attackers of $a$, $\{c_1, \dots , c_m \}$ is the set of all (necessary) supporters of $a$, and we have that
	\begin{align*}
		\mathbf{N_b^P}
		&=N(\|b_1\|\ast\|r_{b_1}^a\|)\ast\dots\ast N(\|b_k\|\ast\|r_{b_k}^a\|)\\
		&=\prod_{i=1}^k(1-\|b_i\|\|r_{b_i}^a\|)\\
	\end{align*}
	and
	\begin{align*}
		\mathbf{N_c^P}
		&=N(N(\|c_1\|)\ast\|t_{c_1}^a\|)\ast\dots\ast N(N(\|c_m\|)\ast\|t_{c_m}^a\|)\\
		&=\prod_{i=1}^m(1-(1-\|c_i\|)\|t_{c_i}^a\|)\\
		&=\prod_{i=1}^m(1-\|t_{c_i}^a\|+\|c_i\|\|t_{c_i}^a\|).\\
	\end{align*}
	Thus we have that
	\begin{align*}
		\mathbf{N_b^P}\ast \mathbf{N_c^P}=\prod_{i=1}^k(1-\|b_i\|\|r_{b_i}^a\|)\prod_{i=1}^m(1-\|t_{c_i}^a\|+\|c_i\|\|t_{c_i}^a\|).
	\end{align*}
	
	Thus, an assignment $\|\cdot\|$ is a model of $ec_{HN}(HAFS)$ in $\mathcal{PL}_{[0, 1]}^P$
	\\$\Longleftrightarrow$ it is a solution of the equational system: $\forall a\in U$,
	\begin{equation*}
		\|a\|=\prod_{i=1}^k(1-\|b_i\|\|r_{b_i}^a\|)\prod_{i=1}^m(1-\|t_{c_i}^a\|+\|c_i\|\|t_{c_i}^a\|),
	\end{equation*}
	where $\{b_1, \dots , b_k \}$ is the set of all attackers of $a$, and $\{c_1, \dots , c_m \}$ is the set of all (necessary) supporters of $a$
	\\$\Longleftrightarrow$ it is a model of the $HAFS$ under equational semantics $Eq^P$.
\end{proof}

From this theorem, we have the corollary below. 
\begin{cor}
	The equational semantics $Eq^P$ is translatable via $ec_{HN}$ and $\mathcal{PL}_{[0, 1]}^P$, i.e., it is equivalent to a fuzzy normal encoded semantics associated with $\mathcal{PL}_{[0, 1]}^P$.
\end{cor}

By letting the $\mathcal{PL}_{[0, 1]}$ in Theorem \ref{thm7} be $\mathcal{PL}_{[0, 1]}^L$ (equipped with the standard negation, the {\L}ukasiewicz t-norm $\ast$, and the R-implication $I_\ast$), we obtain an equational system $Eq^L$ as an instantiation of $Eq^{ec_{HN}}$. The corresponding model equivalence theorem follows from this specification. Since the equational system \(Eq^L\) has a complex specific form, we refrain from presenting its details. Interested readers can derive it by themselves.

\subsection{Relationships between complete semantics and fuzzy normal encoded semantics $Eq^{ec_{HN}}$}
Exploring the relationships between complete semantics and fuzzy normal encoded semantics can help us find the connections between 3-valued semantics and [0,1]-valued semantics. It can also help us choose proper $\mathcal{PL}_{[0,1]}$s for encoding when we want to align the fuzzy normal encoded semantics with the complete semantics in an expected way. Additionally, the model equivalence between the complete semantics and fuzzy normal encoded semantics (with a technical treatment) gives us a new approach to calculating the models of the complete semantics.
\subsubsection{The relationship (I) between complete semantics and $Eq^{ec_{HN}}$ with zero-divisor free t-norms}
Suppose that the $\mathcal{PL}_{[0,1]}^\circledcirc$ in Theorem \ref{thm7} is equipped with a negation $N$, a zero-divisor-free t-norm $\circledcirc$ and an R-implication $I_\circledcirc$. We denote the corresponding encoded semantics by $Eq_\circledcirc^{ec_{HN}}$: $\forall a\in U$,
\begin{equation*}
	\|a\| =\mathbf{N_b}\circledcirc \mathbf{N_c},
\end{equation*}
where $\mathbf{N_b}=N(\|b_1\|\circledcirc\|r_{b_1}^a\|)\circledcirc\dots\circledcirc N(\|b_k\|\circledcirc\|r_{b_k}^a\|)$, $\mathbf{N_c}=N(N(\|c_1\|)\circledcirc\|t_{c_1}^a\|)\circledcirc\dots\circledcirc N(N(\|c_m\|)\circledcirc\|t_{c_m}^a\|)$, $\{b_1, \dots , b_k \}$ is the set of all attackers of $a$, and $\{c_1, \dots , c_m \}$ is the set of all (necessary) supporters of $a$. 

\begin{thm}\label{thm16}
	For an $HAFS$ and an encoded semantics $Eq_\circledcirc^{ec_{HN}}$, if an assignment $\|\cdot\|$ is a model of the $HAFS$ under the $Eq_\circledcirc^{ec_{HN}}$, then $T_3(\|\cdot\|)$ is a model of the $HAFS$ under the complete semantics.
\end{thm} 

\begin{proof}
	For a given $HAFS$, for any $a\in U$, let $\{b_1, b_2, \dots, b_k\}$ be the set of all attackers of $a$ and $\{c_1, \dots , c_m \}$ be the set of all (necessary) supporters of $a$. If an assignment $\|\cdot\|$ is a model of the $HAFS$ under the $Eq_\circledcirc^{ec_{HN}}$, then denoting $T_3(\|\cdot\|)$ by $\|\cdot\|_3$, we have  
	\begin{equation*}
		\|a\|_3=
		\begin{cases}
			1 & \text{iff } \|a\|=1, \\
			0 & \text{iff } \|a\|=0,\\
			\frac{1}{2} & \text{otherwise}.
		\end{cases}
	\end{equation*}
	Then we need to discuss three cases.
	\begin{itemize}
		\item Case 1, $\|a\|_3=0$.\\
		$\|a\|_3=0$
		\\$\Longleftrightarrow$ $\|a\|=0$ 
		\\$\Longleftrightarrow$ $N(\|b_1\|\circledcirc\|r_{b_1}^a\|)\circledcirc\dots\circledcirc N(\|b_k\|\circledcirc\|r_{b_k}^a\|)\circledcirc N(N(\|c_1\|)\circledcirc\|t_{c_1}^a\|)\circledcirc\dots\circledcirc N(N(\|c_m\|)\circledcirc\|t_{c_m}^a\|)=0$
		\\$\Longleftrightarrow$  
		$\exists (b,a)\in R: N(\|b\|\circledcirc\|r_{b}^a\|)=0$ or $\exists (c,a)\in T: N(N(\|c\|)\circledcirc\|t_{c}^a\|)=0$
		\\$\Longleftrightarrow$
		$\exists (b,a)\in R: \|b\|\circledcirc\|r_{b}^a\|=1$ or $\exists (c,a)\in T: N(\|c\|)\circledcirc\|t_{c}^a\|=1$
		\\$\Longleftrightarrow$
		$\exists (b,a)\in R: \|b\|=\|r_{b}^a\|=1$ or $\exists (c,a)\in T: \|c\|=0$ and $\|t_{c}^a\|=1$
		\\$\Longleftrightarrow$
		$\exists (b,a)\in R: \|b\|_3=\|r_{b}^a\|_3=1$ or $\exists (c,a)\in T: \|c\|_3=0$ and $\|t_{c}^a\|_3=1$
		\\$\Longleftrightarrow$ for the assignment $\|\cdot\|_3$, $\|a\|_3=0$ satisfies complete semantics.
		
		\item Case 2, $\|a\|_3=1$.\\
		$\|a\|_3=1$
		\\$\Longleftrightarrow$ $\|a\|=1$
		\\$\Longleftrightarrow$ $N(\|b_1\|\circledcirc\|r_{b_1}^a\|)\circledcirc\dots\circledcirc N(\|b_k\|\circledcirc\|r_{b_k}^a\|)\circledcirc N(N(\|c_1\|)\circledcirc\|t_{c_1}^a\|)\circledcirc\dots\circledcirc N(N(\|c_m\|)\circledcirc\|t_{c_m}^a\|)=1$
		\\$\Longleftrightarrow$ 
		$N(\|b_1\|\circledcirc\|r_{b_1}^a\|)=\dots= N(\|b_k\|\circledcirc\|r_{b_k}^a\|)= N(N(\|c_1\|)\circledcirc\|t_{c_1}^a\|)=\dots= N(N(\|c_m\|)\circledcirc\|t_{c_m}^a\|)=1$
		\\$\Longleftrightarrow$ 
		$\|b_1\|\circledcirc\|r_{b_1}^a\|=\dots= \|b_k\|\circledcirc\|r_{b_k}^a\|= N(\|c_1\|)\circledcirc\|t_{c_1}^a\|=\dots= N(\|c_m\|)\circledcirc\|t_{c_m}^a\|=0$
		\\$\Longleftrightarrow$ 		
		$[\forall (b,a)\in R: \|b\|=0 \text{ or } \|r_{b}^a\|=0] \text{ and } [\forall (c,a)\in T: \|c\|=1 \text{ or } \|t_{c}^a\|=0]$
		\\$\Longleftrightarrow$ 		
		$[\forall (b,a)\in R: \|b\|_3=0 \text{ or } \|r_{b}^a\|_3=0] \text{ and } [\forall (c,a)\in T: \|c\|_3=1 \text{ or } \|t_{c}^a\|_3=0]$		
		\\$\Longleftrightarrow$ for the assignment $\|\cdot\|_3$, $\|a\|_3=1$ satisfies complete semantics.
		
		\item Case 3, $\|a\|_3=\frac{1}{2}$.\\
		$\|a\|_3=\frac{1}{2}$
		\\$\Longleftrightarrow$ not the case $\|a\|_3=1$ or $\|a\|_3=0$
		\\$\Longleftrightarrow$ for the assignment $\|\cdot\|_3$, neither $\|a\|_3=1$ nor $\|a\|_3=0$ satisfies the complete semantics
		\\$\Longleftrightarrow$ for the assignment $\|\cdot\|_3$, $\|a\|_3=\frac{1}{2}$ satisfies complete semantics.
	\end{itemize}
	From the three cases, for a given model $\|\cdot\|$ of an $HAFS$ under $Eq_\circledcirc^{ec_{HN}}$, $T_3(\|\cdot\|)$ is a model of the $HAFS$ under the complete semantics.
\end{proof}

This theorem shows that each model of an $HAFS$ under an encoded semantics $Eq_\circledcirc^{ec_{HN}}$ can be turned into a model of the $HAFS$ under the complete semantics by the ternarization function $T_3$. Since $\mathcal{PL}_{[0,1]}^P$ is a particular $\mathcal{PL}_{[0,1]}^\circledcirc$ and equational semantics $Eq^P$ is a special encoded semantics $Eq_\circledcirc^{ec_{HN}}$, from Theorem \ref{thm16}, we have the corollary below.
\begin{cor}
	For an $HAFS$ and the encoded semantics $Eq^P$, if an assignment $\|\cdot\|$ is a model of the $HAFS$ under the $Eq^P$, then $T_3(\|\cdot\|)$ is a model of the $HAFS$ under the complete semantics.
\end{cor} 
\subsubsection{The relationship (II) between complete semantics and $Eq^{ec_{HN}}$ with $\frac{1}{2}$-idempotent t-norms}
Next we explore the situation that a model of an $HAFS$ under complete semantics is a model of the $HAFS$ under a characterized encoded semantics. We give the characterized t-norm below.
\begin{defn}
	An \emph{$\frac{1}{2}$-idempotent t-norm} is a t-norm $\odot$ that satisfies $\frac{1}{2}\odot \frac{1}{2}=\frac{1}{2}$.
\end{defn}
Suppose that the $\mathcal{PL}_{[0,1]}^\odot$ in Theorem \ref{thm7} is equipped with a standard negation, a $\frac{1}{2}$-idempotent t-norm $\odot$ and an R-implication $I_\odot$. We denote the corresponding encoded semantics by $Eq_\odot^{ec_{HN}}$: $\forall a\in U$,
\begin{equation}\label{odot}
	\|a\| =\mathbf{N_b^\odot}\odot \mathbf{N_c^\odot},
\end{equation}
where $\mathbf{N_b^\odot}=N(\|b_1\|\odot\|r_{b_1}^a\|)\odot\dots\odot N(\|b_k\|\odot\|r_{b_k}^a\|)$, $\mathbf{N_c^\odot}=N(N(\|c_1\|)\odot\|t_{c_1}^a\|)\odot\dots\odot N(N(\|c_m\|)\odot\|t_{c_m}^a\|)$, $\{b_1, \dots , b_k \}$ is the set of all attackers of $a$, and $\{c_1, \dots , c_m \}$ is the set of all (necessary) supporters of $a$. 

\begin{thm}\label{idem}
	For an $HAFS$ and an encoded semantics $Eq_\odot^{ec_{HN}}$, if an assignment $\|\cdot\|$ is a model of the $HAFS$ under the complete semantics, then $\|\cdot\|$ is a model of the $HAFS$ under the $Eq_\odot^{ec_{HN}}$.
\end{thm} 

\begin{proof}
	For a given $HAFS$ and for any $a\in U$, let $\{b_1, \dots , b_k \}$ be the set of all attackers of $a$ and $\{c_1, \dots , c_m \}$ be the set of all (necessary) supporters of $a$. If an assignment $\|\cdot\|$ is a model of the $HAFS$ under the complete semantics, then we have
	\begin{itemize}
		\item 	$\|a\|=1$ iff [$\forall (b,a)\in R: \|b\|=0$ or $\|r_{b}^a\|=0$] and [$\forall (c,a)\in T: \|c\|=1$ or $\|t_{c}^a\|=0$];
		\item$\|a\|=0$ iff [$\exists (b,a)\in R: \|b\|=1$ and $\|r_{b}^a\|=1$] or [$\exists (c,a)\in T: \|c\|=0$ and $\|t_{c}^a\|=1$];
		\item $\|a\|= \frac{1}{2}$ iff otherwise.
	\end{itemize}
	Then we need to discuss three cases.
	\begin{itemize}
		\item Case 1, $\|a\|=0$.\\
		$\|a\|=0$
		\\$\Longleftrightarrow$ $[\exists (b,a)\in R: \|b\|=1 \text{ and } \|r_{b}^a\|=1] \text{ or } [\exists (c,a)\in T: \|c\|=0 \text{ and } \|t_{c}^a\|=1]$
		\\$\Longleftrightarrow$ $[\exists (b,a)\in R: \|b\|\odot\|r_{b}^a\|=1] \text{ or } [\exists (c,a)\in T: N(\|c\|)\odot\|t_{c}^a\|=1]$
		\\$\Longleftrightarrow$ $[\exists (b,a)\in R: N(\|b\|\odot\|r_{b}^a\|)=0] \text{ or } [\exists (c,a)\in T: N(N(\|c\|)\odot\|t_{c}^a\|)=0]$		
		\\$\Longleftrightarrow$ $\mathbf{N_b^\odot}=N(\|b_1\|\odot\|r_{b_1}^a\|)\odot\dots\odot N(\|b_k\|\odot\|r_{b_k}^a\|)=0$ or $\mathbf{N_c^\odot}=N(N(\|c_1\|)\odot\|t_{c_1}^a\|)\odot\dots\odot N(N(\|c_m\|)\odot\|t_{c_m}^a\|)=0$ (since $\odot$ is a $\frac{1}{2}$-idempotent t-norm and the set of values of all elements in $U$ is $\{0,1,\frac{1}{2}\}$) 
		\\$\Longleftrightarrow$
		$\mathbf{N_b^\odot}\odot\mathbf{N_c^\odot}=0$
		\\$\Longleftrightarrow$ 
		$\|a\|=\mathbf{N_b^\odot}\odot\mathbf{N_c^\odot}$, i.e., Equation \ref{odot} holds under this case.
		\item Case 2, $\|a\|=1$.\\
		$\|a\|=1$
		\\$\Longleftrightarrow$ $[\forall (b,a)\in R: \|b\|=0 \text{ or } \|r_{b}^a\|=0] \text{ and } [\forall (c,a)\in T: \|c\|=1 \text{ or } \|t_{c}^a\|=0]$
		\\$\Longleftrightarrow$ $[\forall (b,a)\in R: \|b\|\odot\|r_{b}^a\|=0] \text{ and } [\forall (c,a)\in T: N(\|c\|)\odot\|t_{c}^a\|=0]$
		\\$\Longleftrightarrow$ $[\forall (b,a)\in R: N(\|b\|\odot\|r_{b}^a\|)=1] \text{ and } [\forall (c,a)\in T: N(N(\|c\|)\odot\|t_{c}^a\|)=1]$
		\\$\Longleftrightarrow$ $\mathbf{N_b^\odot}=N(\|b_1\|\odot\|r_{b_1}^a\|)\odot\dots\odot N(\|b_k\|\odot\|r_{b_k}^a\|)=1$ and $\mathbf{N_c^\odot}=N(N(\|c_1\|)\odot\|t_{c_1}^a\|)\odot\dots\odot N(N(\|c_m\|)\odot\|t_{c_m}^a\|)=1$
		\\$\Longleftrightarrow$
		$\mathbf{N_b^\odot}\odot\mathbf{N_c^\odot}=1$
		\\$\Longleftrightarrow$ 
		$\|a\|=\mathbf{N_b^\odot}\odot\mathbf{N_c^\odot}$, i.e., Equation \ref{odot} holds under this case.
		\item Case 3, $\|a\|=\frac{1}{2}$.\\
		$\|a\|=\frac{1}{2}$
		\\$\Longleftrightarrow$ $\|a\|\neq0$ and $\|a\|\neq1$ 
		\\$\Longleftrightarrow$ from Case 1 and Case 2, $\mathbf{N_b^\odot}\odot\mathbf{N_c^\odot}\neq 0$ and $\mathbf{N_b^\odot}\odot\mathbf{N_c^\odot}\neq 1$
		\\$\Longleftrightarrow$ $\mathbf{N_b^\odot}\odot\mathbf{N_c^\odot}=\frac{1}{2}$
		\\$\Longleftrightarrow$ 
		$\|a\|=\mathbf{N_b^\odot}\odot\mathbf{N_c^\odot}$, i.e., Equation \ref{odot} holds under this case.		
	\end{itemize}
	From the three cases, for a given model $\|\cdot\|$ of an $HAFS$ under the complete semantics, $\|\cdot\|$ is a model of the $HAFS$ under the $Eq_\odot^{ec_{HN}}$.
\end{proof}
\subsubsection{The relationship (III) between complete semantics and $Eq^{ec_{HN}}$ with zero-divisor free $\frac{1}{2}$-idempotent t-norms}
Suppose that the $\mathcal{PL}_{[0,1]}^\circleddash$ in Theorem \ref{thm7} is equipped with a standard negation, a $\frac{1}{2}$-idempotent t-norm without zero divisors and an R-implication. We denote the associated encoded semantics by $Eq_\circleddash^{ec_{HN}}$. Then from Theorem \ref{thm16} and Theorem \ref{idem}, we have a corollary below.
\begin{cor}
	For an $HAFS$, we have:
	\begin{align*}
			&\Big\{\, \|\cdot\| \ \Big|\ \|\cdot\| \text{ is a model under complete semantics} \Big\} \\
		= &\Big\{\, \|\cdot\|_3 \ \Big|\ \|\cdot\| \text{ is a model under } Eq_{\circleddash}^{ec_{HN}} \Big\}.
	\end{align*}
\end{cor}
Since $\mathcal{PL}_{[0,1]}^G$ is a particular $\mathcal{PL}_{[0,1]}^\circleddash$ and equational semantics $Eq^G$ is a special encoded semantics $Eq_\circleddash^{ec_{HN}}$, we have the corollary below.
\begin{cor}
	For an $HAFS$, we have:
	\begin{align*}
		&\Big\{\, \|\cdot\| \ \Big|\ \|\cdot\| \text{ is a model under complete semantics} \Big\} \\
		= &\Big\{\, \|\cdot\|_3 \ \Big|\ \|\cdot\| \text{ is a model under } Eq^G \Big\}.
	\end{align*}
\end{cor}
This establishes an exact correspondence between the complete semantics and the encoded semantics \(Eq_\circleddash^{ec_{HN}}\)
\section{Conclusion}
This paper addresses critical limitations in existing higher-order bipolar $AF$s and their logical encodings, proposing a systematic solution through the $HAFS$ and its associated semantic and encoding methodologies. Below is a concise summary of the research contributions, theoretical implications, and future directions:

\subsection{Summary of Core Contributions}
First, at the syntactic level, we extend the $HOAFN$ to the $HAFS$, breaking the long-standing restriction that only arguments can act as attackers or supporters. By explicitly allowing attacks and supports to serve as both targets and sources of interactions (i.e., $U = A \cup R \cup T$ and $R, T \subseteq U \times U$), the $HAFS$ enables modeling of complex dialectical scenarios where higher-order relations (attacks on attacks, supports for supports) are essential, enhancing both theoretical uniformity and real-world applicability.

Second, in terms of semantics, we establish a coherent and expressive semantic suite for $HAFS$s, encompassing three complementary types: extension-based semantics (conflict-free, admissible, complete, preferred, stable, and grounded extensions), adjacent complete labelling semantics (a 3-valued semantics with labels $0, 1, \frac{1}{2}$), and numerical equational semantics ($[0,1]$-valued fuzzy semantics). This suite unifies discrete three-valued reasoning and continuous fuzzy reasoning, addressing the gap between multi-valued and fuzzy semantics in existing higher-order bipolar $AF$s. Key theoretical results include the equivalence between extension-derived complete labellings and adjacent complete labellings for support-acyclic $HAFS$s (Theorem 2), and the mutual translatability between 3-valued complete semantics and fuzzy equational semantics under specific t-norms (Theorems 8–9 and Corollaries 4–5).

Third, regarding logical encoding, we develop a normal encoding function $ec_{HN}$ that bridges $HAFS$s to propositional logic systems ($\mathcal{PLS}$s). We prove that $HAFS$s under adjacent complete labelling semantics are semantically equivalent to Łukasiewicz’s three-valued logic ($\mathcal{PL}_3^L$) theories (Theorem 3), and $HAFS$s under equational semantics can be encoded into fuzzy $\mathcal{PLS}$s (e.g., $\mathcal{PL}_{[0,1]}^G$, $\mathcal{PL}_{[0,1]}^P$, $\mathcal{PL}_{[0,1]}^L$) via equational systems (Theorems 4, 6, 7). This encoding methodology avoids the syntactic overhead of first-order or modal logics, enabling seamless interoperability with lightweight computational solvers and providing a rigorous logical foundation for $HAFS$ semantics.

\subsection{Theoretical and Practical Implications}
Theoretically, this work advances the formalization of higher-order argumentation by: (1) generalizing $RAFN$ semantics to $HAFS$s, ensuring that the validity of attacks/supports is independent of their sources (aligning with classical logical intuitions); (2) resolving the treatment of support cycles by allowing extensions to include elements in valid support cycles (contrasting with existing approaches that outright reject such elements); (3) establishing a unified framework for linking discrete and continuous argumentation semantics, facilitating cross-paradigm generalization.

Practically, the proposed encoding methodology translates $HAFS$s into computationally tractable logical formulas, enabling the use of existing $\mathcal{PLS}$ solvers to compute argument acceptability. This overcomes the impracticality of complex logical encodings (e.g., first-order or modal logics) in real-world applications, making higher-order bipolar argumentation accessible for scenarios such as multi-agent reasoning, legal dispute resolution, and risk assessment.

\subsection{Limitations and Future Work}
Despite its contributions, this work has several limitations that merit further exploration. For example, this study focuses on static $HAFS$s; dynamic extensions (e.g., adding/removing arguments/relations) and their logical encodings remain an open issue.

Future work may proceed in three directions: (1) extending the $HAFS$ to handle dynamic changes and weighted relations, with applications to real-time decision-making scenarios; (2) conducting empirical evaluations to validate the computational efficiency of the proposed encoding methodology using $\mathcal{PLS}$ solvers (e.g., fuzzy logic toolboxes or three-valued logic reasoners); (3) exploring cross-fertilization with other logical systems (e.g., adaptive logics or paraconsistent logics) to enhance the handling of inconsistency and uncertainty in complex argumentation.

In conclusion, this paper lays the groundwork for a unified theory of higher-order bipolar argumentation, bridging syntactic expressiveness, semantic coherence, and logical computability. The $HAFS$ and its encoding methodologies not only address foundational gaps in argumentation theory but also provide a practical tool for modeling and reasoning about complex, uncertain dialectical interactions.



\bibliographystyle{elsarticle-num} 
\bibliography{ref}



\end{document}